\documentclass[12pt]{amsproc}

\usepackage{newclude}

\usepackage{amssymb}
\usepackage{amsfonts}
\usepackage[latin1]{inputenc}
\usepackage[mathscr]{euscript}
\usepackage{enumerate}
\usepackage[all]{xy}
\usepackage{tkz-graph}
\usetikzlibrary{arrows}
\usetikzlibrary{positioning}
\usepackage{fullpage}
\usepackage{mathtools}

\newtheorem{theorem}{Theorem}[section]
\newtheorem{lemma}[theorem]{Lemma}
\newtheorem{proposition}[theorem]{Proposition}

\theoremstyle{definition}
\newtheorem{definition}[theorem]{Definition}
\newtheorem{example}[theorem]{Example}

\theoremstyle{remark}
\newtheorem{remark}[theorem]{Remark}

\numberwithin{equation}{section}

\newcommand{\dmal}[1]{\begin{align*}{#1}\end{align*}}              
\newcommand{\newf}[3]{{#1}:{#2}\longrightarrow {#3}}					 
\newcommand{\newfd}[5]{\begin{array}{cccc}{#1}: &{#2} &\longrightarrow &{#3} \\& {#4} &\longmapsto &{#5}\end{array}}

\newcommand{\la}[1]{\text{$\mathcal{#1}$}}
\newcommand{\lb}[1]{\text{$\mathscr{#1}$}}

\newcommand{\powerset}[1]{\text{$\lb{P}(#1)$}}								
\newcommand{\eword}{\text{$\omega$}}											

\newcommand{\xia}{\text{$\xi^\alpha$}}											

\newcommand{\uset}[1]{\text{$\uparrow\hspace{-0.1cm}{#1}$}}				
\newcommand{\usetr}[2]{\text{$\uparrow_{\scriptscriptstyle{#2}}\hspace{-0.1cm}{#1}$}} 
\newcommand{\lset}[1]{\text{$\downarrow\hspace{-0.1cm}{#1}$}}				
\newcommand{\lsetr}[2]{\text{$\downarrow_{\scriptscriptstyle{#2}}\hspace{-0.1cm}{#1}$}} 

\newcommand{\dgraphg}[1]{\text{$\lb{#1}$}}									
\newcommand{\dgraphupleg}[3]{\text{$\dgraphg{#1}=(\dgraphg{#1}^0,\dgraphg{#1}^1,#2,#3)$}}	

\newcommand{\dgraph}{\text{$\dgraphg{E}$}}									
\newcommand{\dgraphuple}{\text{$\dgraphupleg{E}{r}{s}$}}					
\newcommand{\alfg}[1]{\text{$\lb{#1}$}}										
\newcommand{\acfg}[1]{\text{$\lb{#1}$}}										
\newcommand{\lbfg}[1]{\text{$\lb{#1}$}}										
\newcommand{\acfrg}[1]{\text{$\acfg{B}_{#1}$}}								

\newcommand{\alf}{\text{$\alfg{A}$}}											
\newcommand{\acf}{\text{$\acfg{B}$}}											
\newcommand{\lbf}{\text{$\lbfg{L}$}}											
\newcommand{\acfra}{\text{$\acfg{B}_{\alpha}$}}								
\newcommand{\lgraphg}[2]{\text{$(\dgraphg{#1},\lbfg{#2})$}}				
\newcommand{\lspaceg}[3]{\text{$(\dgraphg{#1},\lbfg{#2},\acfg{#3})$}}

\newcommand{\lgraph}{\text{$\lgraphg{E}{L}$}}								
\newcommand{\lspace}{\text{$\lspaceg{E}{L}{B}$}}							

\newcommand{\awsetg}[2]{\text{$\lbfg{#1}^{#2}$}}	
\newcommand{\awplus}{\text{$\awsetg{L}{\scriptscriptstyle{\geq 1}}$}}								
\newcommand{\awn}[1]{\text{$\awsetg{L}{#1}$}}								
\newcommand{\awstar}{\text{$\awsetg{L}{\ast}$}}							
\newcommand{\awinf}{\text{$\awsetg{L}{\infty}$}}							
\newcommand{\awleinf}{\text{$\awsetg{L}{\scriptscriptstyle{\leq\infty}}$}}					

\newcommand{\fset}[1]{\text{$\textsf{#1}$}}										
\newcommand{\filt}{\text{$\fset{F}$}}											
\newcommand{\filtw}[1]{\text{$\fset{F}_{#1}$}}

\newcommand{\ftight}{\text{$\textsf{T}$}}										
\newcommand{\ftightw}[1]{\text{$\textsf{T}_{#1}$}}

\newcommand{\ftg}[1]{\text{$\la{#1}$}}											
\newcommand{\ft}{\text{$\ftg{F}$}}												

\newcommand{\hs}{\la{H}}

\newcommand{\sumtl}[3]{\sum_{\substack{#1 \\ #2}}{#3}}

\newcommand{\card}[1]{\# #1}

\allowdisplaybreaks

\begin{document}

\title[Diagonal C*]{C*-algebras of labelled spaces and their diagonal C*-subalgebras}

\author[G. Boava \and G. de Castro \and F. Mortari]{Giuliano Boava \and Gilles G. de Castro \and Fernando de L. Mortari}
\address{Departamento de Matemática, Universidade Federal de Santa Catarina, 88040-970 Florianópolis SC, Brazil.}
\email{g.boava@ufsc.br \\ gilles.castro@ufsc.br \\ \newline fernando.mortari@ufsc.br}


\begin{abstract}
Motivated by Exel's inverse semigroup approach to combinatorial C*-algebras, in a previous work the authors defined an inverse semigroup associated with a labelled space. We construct a representation of the C*-algebra of a labelled space, inspired by how one might cut or glue labelled paths together, that proves that non-zero elements in the inverse semigroup correspond to non-zero elements in the C*-algebra. We also show that the spectrum of its diagonal C*-subalgebra is homeomorphic to the tight spectrum of the inverse semigroup associated with the labelled space.
\end{abstract}

\maketitle

\section{Introduction}
Many C*-algebras are defined in terms of partial isometries and relations. The motivation for these relations often comes from a kind of combinatorial object such as a graph \cite{MR1432596} or a 0-1 matrix \cite{MR561974}. A set of partial isometries in a C*-algebra is generally not closed under multiplication; however, in the examples cited above, one can find an inverse semigroup of partial isometries in the C*-algebra. In \cite{MR2419901}, Exel found a topological space on which this inverse semigroup acts, the \emph{tight spectrum}, and by constructing the C*-algebra of the groupoid of germs of this action, we arrive at the original C*-algebra.

Among the combinatorial objects mentioned above are \emph{labelled spaces}, which were introduced by Bates and Pask in \cite{MR2304922}. The C*-algebras associated with labelled spaces generalize several others such as graph algebras, ultragraph algebras and Carlsen-Matsumoto algebras \cite{MR2304922}. These C*-algebras were further studied in \cite{ETS:9992065}, \cite{MR2542653}, \cite{MR3231477}, \cite{MR2834773}, \cite{MR2873859}. It is interesting to notice that the definition given in \cite{MR2304922} was later changed in \cite{ETS:9992065} because certain projections in the C*-algebra could turn out to be zero when it was desired for them not to be.

In \cite{Boava2016}, the authors applied Exel's construction \cite{MR2419901} to an inverse semigroup defined from a labelled space with multiplication inspired by the relations defining the C*-algebra of the labelled space. The tight spectrum was characterized in Theorem 6.7 of \cite{Boava2016}. In the particular case of a labelled space defined from a graph as in \cite{MR2304922}, the authors found \cite[Proposition 6.9]{Boava2016} that the tight spectrum is homeomorphic to the boundary path space of the underlying graph (studied by Webster in \cite{MR3119197}).

Webster shows that the boundary path space of a graph is the spectrum of a certain commutative C*-subalgebra of the graph C*-algebra called the \emph{diagonal C*-subalgebra}  \cite{MR3119197}. There is a natural generalization of the diagonal C*-subalgebra in the case of labelled spaces. By comparing it with the case of graphs, one would expect the spectrum of this new diagonal C*-subalgebra to be the tight spectrum found by the authors in \cite{Boava2016}. Our main goal is to show that this is actually the case.

We begin the paper by recalling the necessary definitions and results in Section 2, and then, in Section 3, we present the definition of the C*-algebra associated with a labelled space and some of its properties. Section 4 is a rather technical one that prepares for the building of a representation of our new version of the C*-algebra. The representation is constructed in Section 5 and, as a corollary, we show that non-zero elements in the inverse semigroup correspond to non-zero elements in the C*-algebra. In Section 6, we show the main result of this paper, namely that the spectrum of the diagonal C*-subalgebra is homeomorphic to the tight spectrum of the inverse semigroup. 

This work took place before the final version of \cite{ETS:9992065} was published; in the preprint available at the time, the definition of the C*-algebra of a labelled space was different, and one of the motivations of this work was to show that the definition needed revision - based on our description of the tight spectrum given in \cite{Boava2016}. We ended up with a definition for the C*-algebra of a labelled space that is equivalent to the one published in \cite{ETS:9992065}, as Carlsen later pointed out to us. In section 7, we give an example to show that the new definition gives a non-trivial quotient of the C*-algebra defined in the preprint of \cite{ETS:9992065}.

\section{Preliminaries}
\label{section:preliminaries}

We begin by introducing notation and terminology, followed by a review of results and techniques developed in \cite{Boava2016}, to be used throughout the text.

\subsection{Labelled spaces}
\label{subsection:labelled.spaces}

A \emph{(directed) graph} $\dgraphuple$ consists of non-empty sets $\dgraph^0$ (of \emph{vertices}), $\dgraph^1$ (of \emph{edges}), and \emph{range} and \emph{source} functions $r,s:\dgraph^1\to \dgraph^0$.

We say that $v\in\dgraph^0$ is a \emph{source} if $r^{-1}(v)=\emptyset$, a \emph{sink} if $s^{-1}(v)=\emptyset$ and an \emph{infinite emitter} if $s^{-1}(v)$ is an infinite set. Also, $v$ is \emph{singular} if it is either a sink or an infinite emitter, and \emph{regular} otherwise.

A \emph{path of length $n$} on a graph $\dgraph$ is a sequence $\lambda=\lambda_1\lambda_2\ldots\lambda_n$ of edges such that $r(\lambda_i)=s(\lambda_{i+1})$ for all $i=1,\ldots,n-1$. We write $|\lambda|=n$ for the length of $\lambda$ and regard vertices as paths of length $0$. $\dgraph^n$ stands for the set of all paths of length $n$ and $\dgraph^{\ast}=\cup_{n\geq 0}\dgraph^n$. We extend the range and source maps to $\dgraph^{\ast}$ by defining $s(\lambda)=s(\lambda_1)$ and $r(\lambda)=r(\lambda_n)$ if $n\geq 2$ and $s(v)=v=r(v)$ for $n=0$. Similarly, we define a \emph{path of infinite length} (or an \emph{infinite path}) as an infinite sequence $\lambda=\lambda_1\lambda_2\ldots$ of edges such that $r(\lambda_i)=s(\lambda_{i+1})$ for all $i\geq 1$; for such a path, we write $|\lambda|=\infty$. The set of all infinite paths will be denoted by $\dgraph^{\infty}$.

A \emph{labelled graph} consists of a graph $\dgraph$ together with a surjective \emph{labelling map} $\lbf:\dgraph^1\to\alf$, where $\alf$ is a fixed non-empty set, called an \emph{alphabet}. Elements of $\alf$ are called \emph{letters}. The set of all finite words over $\alf$, together with the \emph{empty word} \eword, is denoted by $\alf^{\ast}$, and $\alf^{\infty}$ stands for the set of all infinite words over $\alf$. 	A labelled graph is said to be \emph{left-resolving} if for each $v\in\dgraph^0$ the restriction of $\lbf$ to $r^{-1}(v)$ is injective.

The labelling map $\lbf$ can be used to produce labelling maps $\lbf:\dgraph^n\to\alf^{\ast}$ for all $n\geq 1$, via $\lbf(\lambda)=\lbf(\lambda_1)\ldots\lbf(\lambda_n)$; similarly, it also gives rise to a map $\lbf:\dgraph^{\infty}\to\alf^{\infty}$ in the obvious fashion. Using these maps, the elements of $\awn{n}=\lbf(\dgraph^n)$ are the \emph{labelled paths of length $n$} and the elements of $\awinf=\lbf(\dgraph^{\infty})$ are the \emph{labelled paths of infinite length}. We set $\awplus=\cup_{n\geq 1}\awn{n}$, $\awstar=\{\eword\}\cup\awplus$, and $\awleinf=\awstar\cup\awinf$.

For a subset $A$ of $\dgraph^0$, let
\begin{equation}\label{eqn:lae1}
	\lbf(A\dgraph^1)=\{\lbf(e)\ |\ e\in\dgraph^1\ \mbox{and}\ s(e)\in A\}.
\end{equation}

If $\alpha$ is a labelled path, a path $\lambda$ on the graph such that $\lbf(\lambda)=\alpha$ is called a representative of $\alpha$. The \emph{length of} $\alpha$, denoted by $|\alpha|$, is the length of any one of its representatives. By definition, $\eword$ is also a labelled path, with $|\eword|=0$.  If $1\leq i\leq j\leq |\alpha|$, let $\alpha_{i,j}=\alpha_i\alpha_{i+1}\ldots\alpha_{j}$ if $j<\infty$ and $\alpha_{i,j}=\alpha_i\alpha_{i+1}\ldots$ if $j=\infty$. If $j<i$ set $\alpha_{i,j}=\eword$. 

We say that a labelled path $\alpha$ is a \emph{beginning of} a labelled path $\beta$ if $\beta=\alpha\beta'$ for some labelled path $\beta'$; also, $\alpha$ and $\beta$ are \emph{comparable} if either one is a beginning of the other.

For $\alpha\in\awstar$ and $A\in\powerset{\dgraph^0}$ (where $\powerset{\dgraph^0}$ denotes the power set of $\dgraph^0$), the \emph{relative range of $\alpha$ with respect to} $A$, denoted by $r(A,\alpha)$, is the set
\[r(A,\alpha)=\{r(\lambda)\ |\ \lambda\in\dgraph^{\ast},\ \lbf(\lambda)=\alpha,\ s(\lambda)\in A\}\]
if $\alpha\in\awplus$ and $r(A,\eword)=A$ if $\alpha=\eword$. The \emph{range of $\alpha$}, denoted by $r(\alpha)$, is the set \[r(\alpha)=r(\dgraph^0,\alpha).\]
For $\alpha\in\awplus$ we also define the \emph{source of $\alpha$} as the set \[s(\alpha)=\{s(\lambda)\in\dgraph^0\ |\ \lbf(\lambda)=\alpha\}.\]

It follows that $r(\eword)=\dgraph^0$ and, if $\alpha\in\awplus$, then  $r(\alpha)=\{r(\lambda)\in\dgraph^0\ |\ \lbf(\lambda)=\alpha\}$. The definitions above give maps $s:\awplus\to\powerset{\dgraph^0}$ and $r:\awstar\to\powerset{\dgraph^0}$. Also, if $\alpha,\beta\in\awstar$ are such that $\alpha\beta\in\awstar$ then $r(r(A,\alpha),\beta)=r(A,\alpha\beta)$. Furthermore, for $A,B\in\powerset{\dgraph^0}$ and $\alpha\in\awstar$, it holds that $r(A\cup B,\alpha)=r(A,\alpha)\cup r(B,\alpha)$. Finally, observe that for $A\in\powerset{\dgraph^0}$ one has $\lbf(A\dgraph^1)=\{a\in\alf\ |\ r(A,a)\neq\emptyset\}$.

A \emph{labelled space} is a triple $\lspace$ where $\lgraph$ is a labelled graph and $\acf$ is a family of subsets of $\dgraph^0$ that is \emph{accommodating} for $\lgraph$; that is,  $\acf$ is closed under finite intersections and finite unions, contains all $r(\alpha)$ for $\alpha\in\awplus$, and is \emph{closed under relative ranges}, that is, $r(A,\alpha)\in\acf$ for all $A\in\acf$ and all $\alpha\in\awstar$.

We say that a labelled space $\lspace$ is \emph{weakly left-resolving} if for all $A,B\in\acf$ and all $\alpha\in\awplus$ we have $r(A\cap B,\alpha)=r(A,\alpha)\cap r(B,\alpha)$.

For a given $\alpha\in\awstar$, let \[\acfra=\acf\cap\powerset{r(\alpha)}.\] If $\acf$ is closed under relative complements, then the set $\acfra$ is a Boolean algebra for each $\alpha\in\awplus$, and  $\acfrg{\eword}=\acf$ is a generalized Boolean algebra as in \cite{MR1507106}. By Stone duality  there is a topological space associated with each $\acfra$ with $\alpha\in\awstar$, which we denote by $X_{\alpha}$, consisting of the set of ultrafilters in $\acfra$. 

\subsection{Filters and characters}
\label{subsection:filters.and.characters}

By a \emph{filter} in a partially ordered set $P$ with least element $0$ we mean a subset $\xi$ of $P$ such that
\begin{enumerate}[(i)]
	\item $0\notin\xi$;
	\item if $x\in\xi$ and $x\leq y$, then $y\in\xi$;
	\item if $x,y\in\xi$, there exists $z\in\xi$ such that $z\leq x$ and $z\leq y$.
\end{enumerate} If $P$ is a (meet) semilattice, condition (iii) may be replaced by $x\wedge y\in\xi$ if $x,y\in\xi$.

An \emph{ultrafilter} is a filter which is not properly contained in any filter.

For a given $x\in P$, define \[\uset{x}=\{y\in P \ | \ x\leq y\}\ \ ,\ \  \lset{x}=\{y\in P \ | \ y\leq x\},\] and for subsets $X,Y$ of $P$ define \[\uset{X} = \bigcup_{x\in X}\uset{x} = \{y\in P \ | \ x\leq y \ \mbox{for some} \ x\in X\},\]
and $\usetr{X}{Y} = Y\cap\uset{X}$; the sets $\usetr{x}{Y}$, $\lsetr{x}{Y}$, $\lset{X}$ and $\lsetr{X}{Y}$ are defined analogously.

From now on, $E$ denotes a semilattice with $0$.

\begin{proposition}[\cite{MR2419901}, Lemma 12.3]
	\label{prop:ultrafilter.intersection}
	Let $E$ be a semilattice with 0. A filter $\xi$ in $E$ is an ultrafilter if and only if
	\[\{y\in E \ | \ y\wedge x\neq0 \ \forall\,x\in \xi\}\subseteq \xi.\]
\end{proposition}

A \emph{character} of $E$ is a non-zero function $\phi$ from $E$ to the Boolean algebra $\{0,1\}$ such that
\begin{enumerate}[(a)]
	\item $\phi(0)=0$;
	\item $\phi(x\wedge y)=\phi(x)\wedge\phi(y)$, for all $x,y\in E$.
\end{enumerate}
The set of all characters of $E$ is denoted by $\hat{E}_0$ and we endow $\hat{E}_0$ with the topology of pointwise convergence.

Given $x\in E$, a set $Z\subseteq \lset{x}$ is a \emph{cover} for $x$ if for all non-zero $y\in\lset{x}$, there exists $z\in Z$ such that $z\wedge y\neq 0$.

A character $\phi$ of $E$ is \emph{tight} if for every $x\in E$ and every finite cover $Z$ for $x$, we have
\[\bigvee_{z\in Z}\phi(z)=\phi(x).\]
The set of all tight characters of $E$ is denoted by $\hat{E}_{tight}$, and called the \emph{tight spectrum} of $E$.

We can associate each filter $\xi$ in a semilattice $E$ with a character $\phi_{\xi}$ of $E$ given by
\[\phi_{\xi}(x)=\left\{\begin{array}{lll} 1, & & \mbox{if} \ x\in \xi, \\ 0, & & \mbox{otherwise.} \end{array} \right.\]
Conversely, when $\phi$ is a character, $\xi_{\phi} = \{x\in E \ | \ \phi(x)=1\}$ is a filter in $E$. There is thus a bijection between $\hat{E}_0$ and the set of filters in $E$. We denote by $\hat{E}_{\infty}$ the set of all characters $\phi$ of $E$ such that $\xi_{\phi}$ is an ultrafilter, and a filter $\xi$ in $E$ is said to be \emph{tight} if $\phi_{\xi}$ is a tight character.

It is a fact that every ultrafilter is a tight filter (\cite{MR2419901}, Proposition 12.7), and that $\hat{E}_{tight}$ is the closure of $\hat{E}_{\infty}$ in $\hat{E}_0$ (\cite{MR2419901}, Theorem 12.9).

\subsection{The inverse semigroup of a labelled space}
\label{subsection:inverse.semigroup}

For a given labelled space $\lspace$ \emph{that is weakly left-resolving}, consider the set \[S=\{(\alpha,A,\beta)\ |\ \alpha,\beta\in\awstar\ \mbox{and}\ A\in\acfrg{\alpha}\cap\acfrg{\beta}\ \mbox{with}\ A\neq\emptyset\}\cup\{0\}.\]
A binary operation on $S$ is defined as follows: $s\cdot 0= 0\cdot s=0$ for all $s\in S$ and, given $s=(\alpha,A,\beta)$ and $t=(\gamma,B,\delta)$ in $S$, \[s\cdot t=\left\{\begin{array}{ll}
(\alpha\gamma ',r(A,\gamma ')\cap B,\delta), & \mbox{if}\ \  \gamma=\beta\gamma '\ \mbox{and}\ r(A,\gamma ')\cap B\neq\emptyset,\\
(\alpha,A\cap r(B,\beta '),\delta\beta '), & \mbox{if}\ \  \beta=\gamma\beta '\ \mbox{and}\ A\cap r(B,\beta ')\neq\emptyset,\\
0, & \mbox{otherwise}.
\end{array}\right. \]
If for a given $s=(\alpha,A,\beta)\in S$ we define $s^*=(\beta,A,\alpha)$, then the set $S$, endowed with the operation above, is an inverse semigroup with zero element $0$ (\cite{Boava2016}, Proposition 3.4), whose set of idempotents is \[E(S)=\{(\alpha, A, \alpha) \ | \ \alpha\in\awstar \ \mbox{and} \ A\in\acfra\}\cup\{0\}.\]

The natural order in the semilattice $E(S)$ is described in the next proposition.

\begin{proposition}[\cite{Boava2016} Proposition 4.1]
	\label{prop:order.in.es}
	Let $p=(\alpha, A, \alpha)$ and $q=(\beta, B, \beta)$ be non-zero elements in $E(S)$. Then $p\leq q$ if and only if $\alpha=\beta\alpha'$ and $A\subseteq r(B,\alpha')$.
\end{proposition}

\subsection{Filters in E(S)}

Filters in $E(S)$ are classified in two types: a filter $\xi$ in $E(S)$ is of \emph{finite type} if there exists a word $\alpha\in\awstar$ such that $(\alpha,A,\alpha)\in\xi$ for some $A\in\acfra$ and $\alpha$ has the largest length among all $\beta$ such that $(\beta,B,\beta)\in\xi$ for some $B\in\acfrg{\beta}$. If $\xi$ is not of finite type, then we say it is of \emph{infinite type}.

Let $\xi$ be a filter in $E(S)$ and $p=(\alpha, A,\alpha)$ and $q=(\beta, B, \beta)$ in $\xi$. Since $\xi$ is a filter, then $pq\neq0$; hence, $\alpha$ and $\beta$ are comparable. This says the words of any two elements in a filter are comparable; in particular, for a filter of finite type there is a unique word with largest length associated with it as above.

There is a bijective correspondence between the set of filters of finite type in $E(S)$ and the set of pairs $(\alpha,\ft)$ where $\alpha\in\awstar$ and $\ft$ is a filter in $\acfra$, given by the following two results.

\begin{proposition}[\cite{Boava2016} Proposition 4.3]
	\label{prop:filter.from.finite.word.and.filter}
	Let $\alpha\in\awstar$ and $\ft$ be a filter in $\acfra$. Then
	\[\begin{array}{lll} \xi & = & \displaystyle\bigcup_{A\in\ft}\uset{(\alpha,A,\alpha)} \\
	& = & \{(\alpha_{1,i},A,\alpha_{1,i})\in E(S) \ | \ 0\leq i\leq |\alpha| \ \mbox{and} \ r(A,\alpha_{i+1,|\alpha|})\in\ft\} \end{array}\]
	is a filter of finite type in $E(S)$, with largest word $\alpha$.
\end{proposition}

\begin{proposition}[\cite{Boava2016} Proposition 4.4]
	\label{prop:finite.word.and.filter.from.filter}
	Let $\xi$ be a filter in of finite type in $E(S)$ with largest word $\alpha$, and set
	\[\ft=\{A\in\acf \ | \ (\alpha,A,\alpha)\in\xi\}.\]
	Then $\ft$ is a filter in $\acfra$ and
	\[\xi = \bigcup_{A\in\ft}\uset{(\alpha,A,\alpha)}.\]
\end{proposition}

The situation for filters of infinite type is a bit more involved, for there is no word with largest length in this case: let $\alpha\in\awinf$ and $\{\ftg{F}_n\}_{n\geq0}$ be a family such that $\ftg{F}_n$ is a filter in $\acfrg{\alpha_{1,n}}$ for every $n>0$ and $\ftg{F}_0$ is either a filter in $\acf$ or $\ftg{F}_0=\emptyset$. The family $\{\ftg{F}_n\}_{n\geq0}$ is said to be \emph{admissible for} $\alpha$ if
\[\ftg{F}_n\subseteq\{A\in \acfrg{\alpha_{1,n}} \ | \ r(A,\alpha_{n+1})\in\ftg{F}_{n+1}\}\]
for all $n\geq0$, and is said to be \emph{complete for} $\alpha$ if
\[\ftg{F}_n = \{A\in \acfrg{\alpha_{1,n}} \ | \ r(A,\alpha_{n+1})\in\ftg{F}_{n+1}\}\]
for all $n\geq0$.

Note that any filter $\ftg{F}_n$ in a complete family completely determines all filters that come before it in the sequence. In fact, it can be shown that $\{\ftg{F}_n\}_{n\geq0}$ is complete for $\alpha$ if and only if
\[\ftg{F}_n = \{A\in \acfrg{\alpha_{1,n}} \ | \ r(A,\alpha_{n+1,m})\in\ftg{F}_{m}\}\]
for all $n\geq0$ and all $m>n$ (\cite{Boava2016}, Lemma 4.6).

The following two results give a bijective correspondence between the set of filters of infinite type in $E(S)$ and the set of pairs $(\alpha, \{\ftg{F}_n\}_{n\geq0})$, where $\alpha\in\awinf$ and $\{\ftg{F}_n\}_{n\geq0}$ is a complete family for $\alpha$.

\begin{proposition}[\cite{Boava2016}, Proposition 4.7]
	\label{prop:filter.from.admissible.family}
	Let $\alpha\in\awinf$, $\{\ftg{F}_n\}_{n\geq0}$ be an admissible family for $\alpha$ and define
	\[\xi = \bigcup_{n=0}^{\infty}\bigcup_{A\in\ftg{F}_n}\uset{(\alpha_{1,n},A,\alpha_{1,n})}.\]
	Then $\xi$ is a filter in $E(S)$.
\end{proposition}

\begin{proposition}[\cite{Boava2016}, Proposition 4.8]
	\label{prop:complete.family.from.filter}
	Let $\xi$ be a filter of infinite type in $E(S)$. Then there exists $\alpha\in\awinf$ such that every $p\in\xi$ can be written as $p=(\alpha_{1,n},A,\alpha_{1,n})$ for some $n\geq0$ and some $A\in\acfrg{\alpha_{1,n}}$. Moreover, if we define for each $n\geq0$,
	\[\ftg{F}_n=\{A\in\acf \ | \ (\alpha_{1,n},A,\alpha_{1,n})\in\xi\},\]
	then $\{\ftg{F}_n\}_{n\geq0}$ is a complete family for $\alpha$.
\end{proposition}

Admissible families can be \emph{completed}, in the following sense: if $\alpha\in\awstar$ and $\{\ftg{F}_n\}_{n\geq0}$ is  an admissible family for $\alpha$, then there is a complete family for $\alpha$, say $\{\overline{\ft}_n\}_{n\geq0}$, such that $\ftg{F}_n\subseteq \overline{\ft}_n$, for all $n\geq0$ and both $(\alpha, \{\ftg{F}_n\}_{n\geq0})$ and $(\alpha, \{\overline{\ft}_n\}_{n\geq0})$ generate the same filter in $E(S)$ (\cite{Boava2016}, Proposition 4.11).

One can also talk of admissible and complete families for a labelled path $\alpha$ of \emph{finite} length: the definition is the same, only with a finite family $\{\ftg{F}_n\}_{0\leq n\leq|\alpha|}$ satisfying the above-mentioned conditions. Note then that $\ftg{F}_{|\alpha|}$ determines completely all other filters in the family.

\begin{proposition}[\cite{Boava2016}, Proposition 4.12]
	\label{prop:properties.about.filters.of.finite.type}
	Let $\xi$ be a filter of finite type in $E(S)$ and $(\alpha,\ft)$ be its associated pair. For each $n\in\{0,1,\ldots,|\alpha|\}$, define
	\[\ftg{F}_{n}= \{A\in\acf \ | \ (\alpha_{1,n},A,\alpha_{1,n})\in\xi\}.\]
	Then $\ftg{F}_{|\alpha|}=\ft$ and $\{\ftg{F}_n\}_{0\leq n\leq|\alpha|}$ is a complete family for $\alpha$.
\end{proposition}

We can thus bring about a common description for filters of finite and infinite type, summarized in the following theorem.

\begin{theorem}[\cite{Boava2016}, Theorem 4.13]\label{thm:filters.in.E(S)}
	Let $\lspace$ be a labelled space which is weakly left-resolving, and let $S$ be its associated inverse semigroup. Then there is a bijective correspondence between filters in $E(S)$ and pairs $(\alpha, \{\ftg{F}_n\}_{0\leq n\leq|\alpha|})$, where $\alpha\in\awleinf$ and $\{\ftg{F}_n\}_{0\leq n\leq|\alpha|}$ is a complete family for $\alpha$ (understanding that ${0\leq n\leq|\alpha|}$ means $0\leq n<\infty$ when $\alpha\in\awinf$).
\end{theorem}

We may occasionally denote a filter $\xi$ in $E(S)$ with associated labelled path $\alpha\in\awleinf$ by $\xia$ to stress the word $\alpha$; in addition, the filters in the complete family associated with $\xia$ will be denoted by $\xi^{\alpha}_n$ (or $\xi_n$ when there is no risk of confusion about the associated word). Specifically, \[\xi^{\alpha}_n=\{A\in\acf \ | \ (\alpha_{1,n},A,\alpha_{1,n}) \in \xia\}\]
and the family $\{\xi^{\alpha}_n\}_{0\leq n\leq|\alpha|}$ satisfies
\[\xi^{\alpha}_n = \{A\in \acfrg{\alpha_{1,n}} \ | \ r(A,\alpha_{n+1,m})\in\xi^{\alpha}_{m}\}\]
for all $0\leq n<m\leq|\alpha|$.

In what follows, $\filt$ denotes the set of all filters in $E(S)$ and $\filtw{\alpha}$ stands for the subset of $\filt$ of those filters whose associated word is $\alpha\in\awleinf$.

The following is a complete description of the ultrafilters in $E(S)$.

\begin{theorem}[\cite{Boava2016}, Theorem 5.10]
	\label{thm:ultrafilters}
	Let $\lspace$ be a labelled space which is weakly left-resolving, and let $S$ be its associated inverse semigroup. Then the ultrafilters in $E(S)$ are:
	\begin{enumerate}[(i)]
		\item The filters of finite type $\xia$ such that $\xi_{|\alpha|}$ is an ultrafilter in $\acfra$ and for each  $b\in\alf$ there exists $A\in\xi_{|\alpha|}$ such that $r(A,b)=\emptyset$.
		\item The filters of infinite type $\xia$ such that the family $\{\xi_n\}_{n\geq 0}$ is maximal among all complete families for $\alpha$ (that is, if $\{\ftg{F}_n\}_{n\geq 0}$ is a complete family for $\alpha$ such that $\xi_n\subseteq\ftg{F}_n$ for all $n\geq0$, then $\xi_n=\ftg{F}_n$ for all $n\geq0$).
	\end{enumerate}
	Suppose in addition that the accommodating family $\acf$ is closed under relative complements. Then (ii) can be replaced with
	\begin{enumerate}[(i)']
		\setcounter{enumi}{1}
		\item\label{item:ultrafilters.infinite.type.special.case} The filters of infinite type $\xia$ such that $\xi_n$ is an ultrafilter for every $n>0$ and $\xi_0$ is either an ultrafilter or the empty set.
	\end{enumerate}
\end{theorem}

Let $\ftight$ be the set of all tight filters in $E(S)$, endowed with the topology induced from the topology of pointwise convergence of characters, via the bijection between tight characters and tight filters given at the end of Subsection \ref{subsection:filters.and.characters}. Note then that  $\ftight$ is (homeomorphic to) the tight spectrum of $E(S)$, and will be treated as such in the text.

For each $\alpha\in\awstar$ recall from the end of Subsection \ref{subsection:labelled.spaces} that 
\[X_{\alpha}=\{\ft\subseteq\acfra \ | \ \ft \ \mbox{is an ultrafilter in} \ \acfra\}.\]
Also, define
\[X_{\alpha}^{sink}=\{\ft\in X_{\alpha} \ | \ \forall\,b\in\alf, \ \exists\,A\in\ft \ \mbox{such that} \ r(A,b)=\emptyset\}.\]

Suppose the accommodating family $\acf$ to be closed under relative complements. In this case, for every $\alpha,\beta\in\awplus$ such that $\alpha\beta\in\awplus$, the relative range map $\newf{r(\,\cdot\,,\beta)}{\acfrg{\alpha}}{\acfrg{\alpha\beta}}$ is a morphism of Boolean algebras and, therefore, we have its dual morphism \[\newf{f_{\alpha[\beta]}}{X_{\alpha\beta}}{X_{\alpha}}\] given by $f_{\alpha[\beta]}(\ft)=\{A\in\acfra \ | \ r(A,\beta)\in\ft\}$. When $\alpha=\eword$,  if $\ft\in\acfrg{\beta}$ then $\{A\in\acf \ | \ r(A,\beta)\in\ft\}$ is either an ultrafilter in $\acf=\acfrg{\eword}$ or the empty set, and we can therefore consider $\newf{f_{\eword[\beta]}}{X_{\beta}}{X_{\eword}\cup\{\emptyset\}}$.

Employing this new notation, if $\xia$ is a filter in $E(S)$ and $0\leq m<n\leq|\alpha|$, then $\xi_m=f_{\alpha_{1,m}[\alpha_{m+1,n}]}(\xi_n)$.

If we endow the sets $X_{\alpha}$ with the topology given by the convergence of filters stated at the end of Subsection \ref{subsection:filters.and.characters} (this is the pointwise convergence of characters), it is clear that the functions $f_{\alpha[\beta]}$ are continuous. Furthermore, it is easy to see that $f_{\alpha[\beta\gamma]}=f_{\alpha[\beta]}\circ f_{\alpha\beta[\gamma]}$.

Under the hypothesis of $\acf$ being closed under relative complements, the only tight filters of infinite type in $E(S)$ are the ultrafilters of infinite type (\cite{Boava2016}, Corollary 6.2). The next result classifies the tight filters of finite type.

\begin{proposition}[\cite{Boava2016}, Proposition 6.4]
	\label{prop:tight.filters.of.finite.type}
	Suppose the accommodating family $\acf$ to be closed under relative complements and let $\xia$ be a filter of finite type. Then $\xia$ is a tight filter if and only if $\xi_{|\alpha|}$ is an ultrafilter and at least one of the following conditions hold:
	\begin{enumerate}[(a)]
		\item There is a net $\{\ftg{F}_{\lambda}\}_{\lambda\in\Lambda}\subseteq X_{\alpha}^{sink}$ converging to $\xi_{|\alpha|}$.
		\item There is a net $\{(t_{\lambda},\ftg{F}_{\lambda})\}_{\lambda\in\Lambda}$, where $t_{\lambda}$ is a letter in $\alf$ and $\ftg{F}_{\lambda}\in X_{\alpha t_{\lambda}}$ for each $\lambda\in\Lambda$, such that $\{f_{\alpha[t_{\lambda}]}(\ftg{F}_{\lambda})\}_{\lambda\in\Lambda}$ converges to $\xi_{|\alpha|}$ and $\{t_\lambda\}_\lambda$ converges to infinity (that is, for every $b\in\alf$ there is $\lambda_b\in\Lambda$ with $t_{\lambda}\neq b$ for all $\lambda\geq\lambda_b$).
	\end{enumerate}
\end{proposition}

There is a more algebraic description for the tight filters in $E(S)$, given by the following theorem.

\begin{theorem}[\cite{Boava2016}, Theorem 6.7]
	\label{thm:tight.filters.in.es}
	Let $\lspace$ be a labelled space which is weakly left-resolving and whose accommodating family $\acf$ is closed under relative complements, and let $S$ be its associated inverse semigroup. Then the tight filters in $E(S)$ are:
	\begin{enumerate}[(i)]
		\item\label{item:tight.filters.in.es.infinite.type} The ultrafilters of infinite type.
		\item\label{item:tight.filters.in.es.finite.type} The filters of finite type $\xia$ such that $\xi_{|\alpha|}$ is an ultrafilter in $\acfra$ and for each  $A\in\xi_{|\alpha|}$ at least one of the following conditions hold:
		\begin{enumerate}[(a)]
			\item $\lbf(A\dgraph^1)$ is infinite.
			\item There exists $B\in\acfra$ such that $\emptyset\neq B\subseteq A\cap \dgraph^0_{sink}$.
		\end{enumerate}
	\end{enumerate}
\end{theorem}

Finally, the next result shows the relation between the tight spectrum $\ftight$ and the boundary path space of a directed graph (see also \cite{Boava2016}, Example 6.8).

\begin{proposition}[\cite{Boava2016}, Proposition 6.9]
	\label{prop:tight.spectrum.is.boundary.path.space}
	Let $\lgraph$ be a left-resolving labelled graph such that $\dgraph^0$ is a finite set and let  $\acf=\powerset{\dgraph^0}$. Then the tight spectrum $\ftight$ of the labelled space $\lspace$ is homeomorphic to the boundary path space $\partial\dgraph$ of the graph $\dgraph$.
\end{proposition}

\section{C*-algebras of labelled spaces}\label{section:labelled.space.c*-algebras}

In this section, we present the C*-algebra associated with a labelled space, following \cite{ETS:9992065}.

\begin{definition}[\cite{ETS:9992065}, Definition 2.1]\label{def:c*-algebra.labelled.space.v3}
	Let $\lspace$ be a weakly left-resolving labelled space. The \emph{C*-algebra associated with} $\lspace$, denoted by $C^*\lspace$, is the universal $C^*$-algebra generated by projections $\{p_A \ | \ A\in \acf\}$ and partial isometries $\{s_a \ | \ a\in\alf\}$ subject to the relations
	\begin{enumerate}[(i)]
		\item $p_{A\cap B}=p_Ap_B$, $p_{A\cup B}=p_A+p_B-p_{A\cap B}$ and $p_{\emptyset}=0$, for every $A,B\in\acf$;
		\item $p_As_a=s_ap_{r(A,a)}$, for every $A\in\acf$ and $a\in\alf$;
		\item $s_a^*s_a=p_{r(a)}$ and $s_b^*s_a=0$ if $b\neq a$, for every $a,b\in\alf$;
		\item For every $A\in\acf$ for which $0<\card{\lbf(A\dgraph^1)}<\infty$ and there does not exist $B\in\acf$ such that $\emptyset\neq B\subseteq A\cap \dgraph^0_{sink}$,
		\[p_A=\sum_{a\in\lbf(A\dgraph^1)}s_ap_{r(A,a)}s_a^*.\]
	\end{enumerate}
\end{definition}

\begin{remark}\label{remark.condition.iv}
	In \cite{MR2304922}, Bates and Pask originally defined the C*-algebra associated with a labelled space in a different way, with item (iv) above being replaced with
	\begin{enumerate}[(iv)]
	\item For every $A\in\acf$ such that $0<\card{\lbf(A\dgraph^1)}<\infty$,
		\[p_A=\sum_{a\in\lbf(A\dgraph^1)}s_ap_{r(A,a)}s_a^*.\]
	\end{enumerate}
	In \cite[Remark 2.4]{ETS:9992065}, Bates, Pask and Carlsen observed that this definition would lead to zero vertex projections for sinks. Therefore, they proposed a new definition modifying item (iv). In a preprint version of \cite{ETS:9992065}, they proposed
	\begin{enumerate}[(iv)]
		\item For every $A\in\acf$ such that $0<\card{\lbf(A\dgraph^1)}<\infty$ and $A\cap\dgraph^0_{sink}=\emptyset$,
			\[p_A=\sum_{a\in\lbf(A\dgraph^1)}s_ap_{r(A,a)}s_a^*.\]
	\end{enumerate}
	When we begun this work, the published version of \cite{ETS:9992065} containing Definition \ref{def:c*-algebra.labelled.space.v3} was not yet available and, indeed, one of the goals of this work was to point out that Definition \ref{def:c*-algebra.labelled.space.v3} of the C*-algebra associated with a labelled space is more adequate than the previously given definitions. Item (iv) looks like the relation $\sum_{i=1}^n s_is_i^*=1$ in the Cuntz algebra $\mathcal{O}_n$. To classify a definition as adequate or inadequate is sometimes difficult -- what would the criteria be? In \cite{MR2419901}, Exel applied the theory of tight filters associated with inverse semigroups to show that several C*-algebras obtained from combinatorial objects created since the Cuntz algebras are, in fact,  C*-algebras of inverse semigroups acting on their tight spectra. We believe this establishes a good criterion to choose the adequate definition for the C*-algebra of a labelled space, and we expected Definition \ref{def:c*-algebra.labelled.space.v3} to be the adequate one, due to item (ii)(b) of Theorem \ref{thm:tight.filters.in.es}. After the first preprint of this work, Carlsen pointed out to us that our expected definition was indeed equivalent to the one present in the published version of \cite{ETS:9992065}. In section \ref{section:example.definitions.different}, we present an example that shows Definition \ref{def:c*-algebra.labelled.space.v3} is different from that in the preprint version of \cite{ETS:9992065}.
\end{remark}

One of the points of this paper is to validate Definition \ref{def:c*-algebra.labelled.space.v3}. With Proposition \ref{prop:properties.of.c*-labelled.space} and Section \ref{section:representations}, we justify our choice for the inverse semigroup associated with a labelled space, presented in Section \ref{subsection:inverse.semigroup} and, with Section \ref{section:diagonal.c*-algebra}, we show that Definition \ref{def:c*-algebra.labelled.space.v3} is indeed related to the tight spectrum of the aforementioned inverse semigroup.

Now, we develop some basic properties of $C^*\lspace$; most of these were already discussed by Bates and Pask in \cite{MR2304922}.

Let $\lspace$ be a weakly left-resolving labelled space and consider its C*-algebra $C^*\lspace$. For each word $\alpha=a_1a_2\cdots a_n$, define $s_\alpha=s_{a_1}s_{a_2}\cdots s_{a_n}$; we also set $s_{\eword}=1$, where $\eword$ is the empty word.

\begin{remark}
	$C^*\lspace$ is generated by elements $s_a$ and $p_A$. Thus, the elements $s_{\alpha}$ defined above belong to $C^*\lspace$ if $\alpha$ is not the empty word. On the other hand, $s_\eword$ does not belong to $C^*\lspace$ unless it is unital. We work with $s_{\eword}$ to simplify our statements; for example, $s_{\eword}p_{A}s_{\eword}^*$ means $p_{A}$. We never use $s_{\eword}$ alone.
\end{remark}

\begin{proposition}\label{prop:properties.of.c*-labelled.space}
	The following properties are valid in $C^*\lspace$.
	\begin{enumerate}[(i)]
		\item If $\alpha\notin\awstar$, then $s_{\alpha}=0$.
		\item $p_As_\alpha=s_\alpha p_{r(A,\alpha)}$, for every $A\in\acf$ and $\alpha\in\awstar$.
		\item $s_\alpha^*s_\alpha=p_{r(\alpha)}$ and $s_\beta^*s_\alpha=0$ if $\beta$ and $\alpha$ are not comparable, for every $\alpha,\beta\in\awplus$.
		\item For every $\alpha\in\awplus$, $s_{\alpha}$ is a partial isometry.
		\item Let $\alpha,\beta\in\awstar$ and $A\in\acf$. If $s_\alpha p_A s_\beta^*\neq 0$, then $A\cap r(\alpha)\cap r(\beta)\neq\emptyset$ and $s_\alpha p_A s_\beta^* = s_\alpha p_{A\cap r(\alpha)\cap r(\beta)} s_\beta^*$. 
		\item Let $\alpha,\beta,\gamma,\delta\in\awstar$, $A\in\acfrg{\alpha}\cap\acfrg{\beta}$ and $B\in\acfrg{\gamma}\cap\acfrg{\delta}$. Then
		\[(s_\alpha p_A s_\beta^*)(s_\gamma p_B s_\delta^*) = \left\{ 
		\begin{array}{ll}
		s_{\alpha\gamma'}p_{r(A,\gamma')\cap B}s_{\delta}^*, & \mbox{if} \ \gamma=\beta\gamma', \\
		s_{\alpha}p_{A\cap r(B,\beta')}s_{\delta\beta'}^*, & \mbox{if} \ \beta=\gamma\beta', \\
		0, & \mbox{otherwise}. 
		\end{array}\right.\]
		In particular, $s_\alpha p_A s_\beta^*$ is a partial isometry.
		\item Every non-zero finite product of terms of types $s_a$, $p_B$ and $s_b^*$ can be written as $s_\alpha p_A s_\beta^*$, where $A\in \acfrg{\alpha}\cap\acfrg{\beta}$.
		\item $C^*\lspace = \overline{\mathrm{span}}\{s_\alpha p_A s_\beta^* \ | \ \alpha,\beta\in\awstar \ \mbox{and} \ A\in\acfrg{\alpha}\cap\acfrg{\beta}\}$.
		\item The elements of the form $s_\alpha p_A s_\alpha^*$, where $\alpha\in\awstar$, are commuting projections. Furthermore,
		\[(s_\alpha p_A s_\alpha^*)(s_\beta p_B s_\beta^*) = \left\{ 
		\begin{array}{ll}
		s_{\beta}p_{r(A,\beta')\cap B}s_{\beta}^*, & \mbox{if} \ \beta=\alpha\beta', \\
		s_{\alpha}p_{A\cap r(B,\alpha')}s_{\alpha}^*, & \mbox{if} \ \alpha=\beta\alpha', \\
		0, & \mbox{otherwise}. 
		\end{array}\right.\]
	\end{enumerate}
\end{proposition}

\begin{proof}
	\begin{enumerate}[(i)]
		\item If $\alpha = a_1a_2\cdots a_n \notin \awstar$, then $r(\cdots r(r(a_1), a_2)\cdots,a_n)=\emptyset$. Therefore, by using that $s_{a_1}$ is a partial isometry and items (i), (ii) and (iii) of Definition \ref{def:c*-algebra.labelled.space.v3}, we have
		\[s_\alpha = s_{a_1}s_{a_2}\cdots s_{a_n} = s_{a_1}p_{r(a_1)}s_{a_2}\cdots s_{a_n} = s_{a_1}s_{a_2}\cdots s_{a_n}p_{r(\cdots r(r(a_1), a_2)\cdots,a_n)} = 0.\]
		
		\item This is clear from item (ii) of Definition \ref{def:c*-algebra.labelled.space.v3}.
		
		\item The first equality follows by induction using that $s_{ab}^*s_{ab}=s_b^*s_a^*s_as_b = s_b^*p_{r(a)}s_b = s_b^*s_bp_{r(r(a),b)} = p_{r(b)}p_{r(r(a),b)} = p_{r(r(a),b)} = p_{r(ab)}$. To see the second one, suppose $\alpha,\beta\in\awplus$ are not comparable, that is, there are $\alpha',\beta',\gamma\in\awstar$ and $a,b\in\alf$ with $a\neq b$ such that $\alpha=\gamma a \alpha'$ and $\beta=\gamma b\beta'$. Therefore,
		\[s_\beta^*s_\alpha = (s_\gamma s_b s_{\beta'})^*(s_\gamma s_a s_{\alpha'}) = s_{\beta'}^* s_b^* s_\gamma^*  s_\gamma s_a s_{\alpha'} = s_{\beta'}^* s_b^* p_{r(\gamma)} s_a s_{\alpha'} = s_{\beta'}^* s_b^* s_a p_{r(r(\gamma),a)} s_{\alpha'} = 0,\]
		since $s_b^* s_a=0$.
		
		\item It follows by the previous item, since $s_\alpha^*s_\alpha$ is a projection.
		
		\item Since $s_\alpha p_A s_\beta^* = s_\alpha p_{r(\alpha)}p_A p_{r(\beta)} s_\beta^* = s_\alpha p_{r(\alpha)\cap A \cap r(\beta)} s_\beta^*$, then $r(\alpha)\cap A \cap r(\beta)\neq\emptyset$ if $s_\alpha p_A s_\beta^*\neq 0$. 
		
		\item If $\beta$ and $\gamma$ are not comparable, then $(s_\alpha p_A s_\beta^*)(s_\gamma p_B s_\delta^*) = 0$, by item (iii). Now, suppose $\gamma = \beta\gamma'$ (the other case is similar and both cases coincide if $\beta=\gamma$). Then
		\dmal{(s_\alpha p_A s_\beta^*)(s_\gamma p_B s_\delta^*) &= s_\alpha p_A s_\beta^*s_\beta s_{\gamma'} p_B s_\delta^* = s_\alpha p_A p_{r(\beta)} s_{\gamma'} p_B s_\delta^* \\ &= s_\alpha p_A s_{\gamma'} p_B s_\delta^* = s_\alpha s_{\gamma'} p_{r(A,\gamma')}  p_B s_\delta^* = s_{\alpha\gamma'} p_{r(A,\gamma')\cap B} s_\delta^*.}
		Applying this product rule, we see that $s_\alpha p_A s_\beta^*$ is a partial isometry, since
		\[(s_\alpha p_A s_\beta^*)(s_\alpha p_A s_\beta^*)^*(s_\alpha p_A s_\beta^*) = (s_\alpha p_A s_\beta^*)(s_\beta p_A s_\alpha^*)(s_\alpha p_A s_\beta^*) = (s_\alpha p_A s_\alpha^*)(s_\alpha p_A s_\beta^*) = s_\alpha p_A s_\beta^*.\]
		
		\item Consider a non-zero product as in the statement. If we have a $s_b^*$ on the left of a $s_a$, we must have $a=b$, since the product is non-zero. In this case, we can replace $s_a^*s_a$ by $p_{r(a)}$. If we have $p_B$ on the left of a $s_a$, we can replace $p_Bs_a$ by $s_ap_{r(B,a)}$. Similarly, we can replace $s_b^*p_B$ by $p_{r(B,b)}s_b^*$. Applying these replacements whenever possible, we end up with a product like
		\[s_{a_1}s_{a_2}\cdots s_{a_n}p_{B_1}p_{B_2}\cdots p_{B_m} s_{b_k}^*s_{b_{k-1}}^*\cdots s_{b_1}^*.\]
		Taking $\alpha = a_1a_2\cdots a_n$, $A = B_1\cap\cdots\cap B_m$ and $\beta = b_1b_2\cdots b_k$, the product reduces to $s_\alpha p_A s_\beta^*$ and, by item (v), we can suppose $A\in \acfrg{\alpha}\cap\acfrg{\beta}$.
		
		\item Immediate from the previous item.
		
		\item Applying item (vi) to the product $(s_\alpha p_A s_\alpha^*)(s_\beta p_B s_\beta^*)$, we obtain
		\[(s_\alpha p_A s_\alpha^*)(s_\beta p_B s_\beta^*) = \left\{ 
		\begin{array}{ll}
		s_{\beta}p_{r(A,\beta')\cap B}s_{\beta}^*, & \mbox{if} \ \beta=\alpha\beta', \\
		s_{\alpha}p_{A\cap r(B,\alpha')}s_{\alpha}^*, & \mbox{if} \ \alpha=\beta\alpha', \\
		0, & \mbox{otherwise}. 
		\end{array}\right.\]
		Interchanging $\alpha$ and $A$ with $\beta$ and $B$, it is clear that $s_\alpha p_A s_\alpha^*$ commutes with $s_\beta p_B s_\beta^*$. Finally, taking $\beta=\alpha$ and $B=A$, we see that $s_\alpha p_A s_\alpha^*$ is a projection.
	\end{enumerate}
\end{proof}

Consider the subset $R=\{s_\alpha p_A s_\beta^* \ | \ \alpha,\beta\in\awstar \ \mbox{and} \ A\in\acfrg{\alpha}\cap\acfrg{\beta}\}$ of $C^*\lspace$.  Properties (v) to (viii) from the previous proposition say that $R$ is a semigroup whose linear span is dense in $C^*\lspace$. The inverse semigroup in Section \ref{subsection:inverse.semigroup} was defined based on $R$, but they might not be isomorphic: two idempotents of $S$ may give the same element in the C*-algebra, as in Example \ref{example:different.definitions} (there, triples of the form $(a^n,\dgraph^0,a^n)$ are all different for $n\geq 1$, but $s_{a^n}p_{\dgraph^0}s^*_{a^n}=1$  for all $n\geq 1$).

We finish this section presenting a C*-subalgebra of $C^*\lspace$ that is commutative and will be studied in Section \ref{section:diagonal.c*-algebra}.

\begin{definition}\label{def:diagonal.subalgebra}
	Let $\lspace$ be a weakly left-resolving labelled space. The \emph{diagonal C*-algebra associated with} $\lspace$, denoted by $\Delta\lspace$, is the C*-subalgebra of $C^*\lspace$ generated by the elements $s_\alpha p_A s_\alpha^*$, that is,
	\[\Delta\lspace = C^*(\{s_\alpha p_A s_\alpha^* \ | \ \alpha\in\awstar \ \mbox{and} \ A\in\acfrg{\alpha}\}).\]
\end{definition}

By item (ix) of Proposition \ref{prop:properties.of.c*-labelled.space}, $\Delta\lspace$ is an abelian C*-algebra generated by commuting projections and
\[\Delta\lspace = \overline{\mathrm{span}}\{s_\alpha p_A s_\alpha^* \ | \ \alpha\in\awstar \ \mbox{and} \ A\in\acfrg{\alpha}\}.\]

\section{Filter surgery in $E(S)$}
\label{section:filter.surgery}

In this section, we define functions that are going to be used later on to construct a representation of $C^*\lspace$. These functions generalize two operations that can easily be done with paths on a graph $\dgraph$: gluing paths, that is, given $\mu$ and $\nu$ paths on $\dgraph$ such that $r(\mu)=s(\nu)$, it is easy to see that $\mu\nu$ is a new path on $\dgraph$; and cutting paths, that is, given a path $\mu\nu$ on $\dgraph$ then $\nu$ is also a path on the graph.

In the context of labelled spaces, we have an extra layer of complexity because filters in $E(S)$ are described not only by a labelled path but also by a complete family of filters associated with it, as in Theorem \ref{thm:filters.in.E(S)}. When we cut or glue labelled paths, the Boolean algebras where the filters lie change because they depend on the labelled path. We also note that, since we are only interested in tight filters in $E(S)$, the families considered below will consist only of ultrafilters.

Let us begin with the problem of describing new filters by gluing labelled paths. Consider composable labelled paths $\alpha\in\awplus$ and $\beta\in\awstar$, and an ultrafilter $\ft\in X_\beta$; a simple way to produce a subset $\ftg{J}$ of $\acfrg{\alpha\beta}$ from $\ft$ is by cutting the elements of $\ft$ by $r(\alpha\beta)$, that is, \[ \ftg{J}=\{C\cap r(\alpha\beta)\ |\ C\in\ft\}. \]

It may be the case, however, that $C\cap r(\alpha\beta)=\emptyset$ for some $C\in\ft$. Since $r(\alpha\beta)\in\acfrg{\beta}$, by Proposition \ref{prop:ultrafilter.intersection} and using the fact that $\ft$ is an ultrafilter it can be seen that the intersections $C\cap r(\alpha\beta)$ are non-empty for all $C\in\ft$ if and only if $r(\alpha\beta)\in\ft$. The following simple result is useful for the present argument, and a proof is included for convenience.

\begin{lemma}
	\label{lemma:semilattice.ultrafilter.cutdown}
	Given $E$ a meet semilattice with $0$, $y\in E$ and $\ft$ an ultrafilter in $E$, consider \[ \ftg{J}=\{x\wedge y\ |\ x\in\ft\}. \] Then $\ftg{J}$ is an ultrafilter in $\lset{y}$ if and only if $y\in\ft$.
\end{lemma}
\begin{proof}
	Clearly $\ftg{J}\subseteq\lset{y}$ and $\ftg{J}$ is closed by finite meets. Proposition \ref{prop:ultrafilter.intersection} ensures $y\in\ft$ if and only if $x\wedge y\neq 0$ for all $x\in\ft$, and so if $y\notin\ft$ one sees already that $\ftg{J}$ is not a filter.
	Suppose then that $y\in\ft$. Given $x\in\ft$ and $z\in\lset{y}$ such that $x\wedge y\leq z$, since $\ft$ is a filter it follows that $x\wedge y\in\ft$ and thus $z\in\ft$, whence $z=z\wedge y\in\ftg{J}$, showing $\ftg{J}$ is a filter.
	
	To show $\ftg{J}$ is an ultrafilter, let $z\in\lset{y}$ be such that $z\wedge u\neq 0$ for all $u\in\ftg{J}$. Given that $z=y\wedge z$ one has \[ x\wedge z= x\wedge(y\wedge z)=(x\wedge y)\wedge z\neq 0 \] for all $x\in\ft$, hence $z\in\ft$ since $\ft$ is an ultrafilter; but then $z=z\wedge y\in\ftg{J}$, and the result follows.
\end{proof}

\begin{lemma}
	\label{lemma:glue.alpha.ok}
	Let $\alpha\in\awplus$ and $\beta\in\awstar$ be such that $\alpha\beta\in\awplus$, let $\ft\in X_\beta$, and consider \[ \ftg{J}=\{C\cap r(\alpha\beta)\ |\ C\in\ft\}. \] Then $\ftg{J}\in X_{\alpha\beta}$ if and only if $r(\alpha\beta)\in\ft$.
\end{lemma}
\begin{proof}
	Follows immediately from Lemma \ref{lemma:semilattice.ultrafilter.cutdown}, with $E=\acfrg{\beta}$ and $y=r(\alpha\beta)$ (note that $\lset{r(\alpha\beta)}=\acfrg{\alpha\beta}$).
\end{proof}

For labelled paths $\alpha\in\awplus$ and $\beta\in\awstar$ such that $\alpha\beta\in\awplus$, denote by $X_{(\alpha)\beta}$ the set of ultrafilters in $\acfrg{\beta}$ that give rise, via cutdown by $r(\alpha\beta)$, to ultrafilters in $\acfrg{\alpha\beta}$. More precisely, using Lemma \ref{lemma:glue.alpha.ok}, \[X_{(\alpha)\beta}=\{\ft\in X_\beta \ |\ r(\alpha\beta)\in\ft\}.\]

There is thus a map $\newf{g_{(\alpha)\beta}}{X_{(\alpha)\beta}}{X_{\alpha\beta}}$, that associates to each ultrafilter $\ft\in X_{(\alpha)\beta}$, the ultrafilter in $\acfrg{\alpha\beta}$ given by \[g_{(\alpha)\beta}(\ft)=\{C\cap r(\alpha\beta)\ |\ C\in\ft\}. \]
Also, for $\alpha=\eword$ define $X_{(\eword)\beta}=X_\beta$ and let $g_{(\eword)\beta}$ denote the identity function on $X_\beta$.

The following lemmas describe properties of these sets and maps, and how they behave with respect to the maps $\newf{f_{\alpha[\beta]}}{X_{\alpha\beta}}{X_\alpha}$ between ultrafilter spaces, introduced at the end of Section \ref{section:preliminaries}.

\begin{lemma}
	\label{lemma:f-class.preserves.Xalphabetas}
	Let $\alpha\in\awplus$ and $\beta,\gamma\in\awstar$ with $\alpha\beta\gamma\in\awplus$ be given. Then
	\begin{enumerate}[(i)]
		\item\label{item:f-class.preserves.parentheses} $f_{\beta[\gamma]}(X_{(\alpha)\beta\gamma})\subseteq X_{(\alpha)\beta}$;
		\item $f_{\beta[\gamma]}^{-1}(X_{(\alpha)\beta})\subseteq X_{(\alpha)\beta\gamma}$.
	\end{enumerate}
\end{lemma}
\begin{proof}
	To prove (i), given $\ft'\in X_{(\alpha)\beta\gamma}$, one has $r(\alpha\beta\gamma)\in\ft'$, and also \[ f_{\beta[\gamma]}(\ft')=\{C\in\acfrg{\beta}\ |\ r(C,\gamma)\in\ft' \}. \] 
	Since $r(\alpha\beta)\in\acfrg{\beta}$ and $r(r(\alpha\beta),\gamma)=r(\alpha\beta\gamma)\in\ft'$, it follows that $r(\alpha\beta)\in f_{\beta[\gamma]}(\ft')$ and therefore $f_{\beta[\gamma]}(\ft')\in X_{(\alpha)\beta}$.
	
	As for (ii), if $\ft'\in f_{\beta[\gamma]}^{-1}(X_{(\alpha)\beta})$ then $f_{\beta[\gamma]}(\ft')\in X_{(\alpha)\beta}$, and thus $r(\alpha\beta)\in\{C\in\acfrg{\beta}\ |\ r(C,\gamma)\in\ft' \}$. This gives \[ r(\alpha\beta\gamma)=r(r(\alpha\beta),\gamma)\in\ft', \] whence $\ft'\in X_{(\alpha)\beta\gamma}$.
\end{proof}

\begin{lemma}
	\label{lemma:g-class.properties} Let $\alpha\in\awplus$ and $\beta,\gamma\in\awstar$ with $\alpha\beta\gamma\in\awplus$ be given.
	Then,
	\begin{enumerate}[(i)]
		\item\label{item:g-class.domain.containment} $X_{(\alpha\beta)\gamma}\subseteq X_{(\beta)\gamma}$;
		
		\item\label{item:g-class.composition.defined} $g_{(\beta)\gamma}(X_{(\alpha\beta)\gamma})\subseteq X_{(\alpha)\beta\gamma}$;
		
		\item\label{item:g-class.composition.rule}
		If $\ft\in X_{(\alpha\beta)\gamma}$, then \[ g_{(\alpha\beta)\gamma}(\ft)=g_{(\alpha)\beta\gamma}\circ g_{(\beta)\gamma}(\ft). \]
		
		\item\label{item:f-class.and.g-class.compatible}
		Suppose the labelled space $\lspace$ is weakly left-resolving. Then, the following diagram is commutative:
		\[\xymatrix{
			X_{(\alpha)\beta\gamma} \ar[r]^{g_{(\alpha)\beta\gamma}} \ar[d]_{f_{\beta[\gamma]}} & X_{\alpha\beta\gamma} \ar[d]^{f_{\alpha\beta[\gamma]}} \\
			X_{(\alpha)\beta} \ar[r]_{g_{(\alpha)\beta}}       & X_{\alpha\beta}. }
		\]
	\end{enumerate}
\end{lemma}
\begin{proof}
	If $\ft\in X_{(\alpha\beta)\gamma}$ then $r(\alpha\beta\gamma)\in \ft$ and, since $\ft$ is a filter in $\acfrg{\gamma}$, the fact that $r(\alpha\beta\gamma)\subseteq r(\beta\gamma)$ gives $r(\beta\gamma)\in\ft$, whence $\ft\in X_{(\beta)\gamma}$; this proves (i). Also, \[ r(\alpha\beta\gamma)=r(\alpha\beta\gamma)\cap r(\beta\gamma)\in \{C\cap r(\beta\gamma)\ |\ C\in\ft\}=g_{(\beta)\gamma}(\ft), \] from where (ii) follows. 
	
	It is clear that \dmal{g_{(\alpha)\beta\gamma}\circ g_{(\beta)\gamma}(\ft)&=g_{(\alpha)\beta\gamma}(\{C\cap r(\beta\gamma)\ |\ C\in\ft\}) \\
		&=\{(C\cap r(\beta\gamma))\cap r(\alpha\beta\gamma)\ | \ C\in\ft\} \\
		&=\{C\cap r(\alpha\beta\gamma)\ | \ C\in\ft\}=g_{(\alpha\beta)\gamma}(\ft),} which is (iii). As for (iv), Lemma \ref{lemma:f-class.preserves.Xalphabetas} ensures $f_{\beta[\gamma]}$ maps $X_{(\alpha)\beta\gamma}$ into $X_{(\alpha)\beta}$. If $\ft\in X_{(\alpha)\beta\gamma}$, then \[ (g_{(\alpha)\beta}\circ f_{\beta[\gamma]})(\ft)=\{C\cap r(\alpha\beta)\ |\ C\in\acfrg{\beta}\ \text{and}\ r(C,\gamma)\in\ft \} \] and \[ (f_{\alpha\beta[\gamma]}\circ g_{(\alpha)\beta\gamma})(\ft)=\{ D\in\acfrg{\alpha\beta}\ |\ r(D,\gamma)\in g_{(\alpha)\beta\gamma}(\ft)\}.  \]
	Given $C\in\acfrg{\beta}$ such that $r(C,\gamma)\in\ft$, note that $C\cap r(\alpha\beta)\in\acfrg{\alpha\beta}$ and that \[ r(C\cap r(\alpha\beta),\gamma)=r(C,\gamma)\cap r(\alpha\beta\gamma)\in g_{(\alpha)\beta\gamma}(\ft), \] using the weakly left-resolving hypothesis; therefore \[ (g_{(\alpha)\beta}\circ f_{\beta[\gamma]})(\ft)\subseteq (f_{\alpha\beta[\gamma]}\circ g_{(\alpha)\beta\gamma})(\ft), \] and since both terms are ultrafilters, equality follows.
\end{proof}

\begin{lemma}
	\label{lemma:g-class.properties2}
	Let $\alpha\in\awplus$ and $\beta\in\awstar$ be such that $\alpha\beta\in\awplus$. Then,
	\begin{enumerate}[(i)]
		\item\label{item:g-class.preserves.sinks}
		$g_{(\alpha)\beta}(X_{(\alpha)\beta}\cap X_{\beta}^{sink})\subseteq X_{\alpha\beta}^{sink}.$
		
		\item\label{item:g-class.continuous}
		$\newf{g_{(\alpha)\beta}}{X_{(\alpha)\beta}}{X_{\alpha\beta}}$ is continuous.
		
		\item\label{item:Xalphabeta.open}
		$X_{(\alpha)\beta}$ is an open subset of $X_{\beta}$.
	\end{enumerate}
	
\end{lemma}
\begin{proof}
	Given $\ft\in X_{(\alpha)\beta}\cap X_{\beta}^{sink}$, for each $a\in\alf$ there is a $D\in\ft$ with $r(D,a)=\emptyset$; consequently, $D\cap r(\alpha\beta)\in g_{(\alpha)\beta}(\ft)$ is such that $r(D\cap r(\alpha\beta),a)=\emptyset$, hence $g_{(\alpha)\beta}(\ft)\in X_{\alpha\beta}^{sink}$, proving (i).
	
	To prove (ii), consider a net $\{\ft_\lambda\}_\lambda\subseteq X_{(\alpha)\beta}$ that converges to $\ft\in X_{(\alpha)\beta}$. For an arbitrary $D\in g_{(\alpha)\beta}(\ft)$, say $D=C\cap r(\alpha\beta)$ for some $C\in\ft$, the convergence above ensures there is $\lambda_0$ such that $\lambda\geq\lambda_0$ implies $C\in\ft_\lambda$, and therefore $D=C\cap r(\alpha\beta)\in g_{(\alpha)\beta}(\ft_\lambda)$.
	
	If on the other hand $D\in\acfrg{\alpha\beta}\backslash g_{(\alpha)\beta}(\ft)$, using that $g_{(\alpha)\beta}(\ft)$ is an ultrafilter there must be an element of it that does not intersect $D$, that is, there is $C\in\ft$ such that $D\cap(C\cap r(\alpha\beta))=\emptyset$, by Proposition \ref{prop:ultrafilter.intersection}. Since $\acfrg{\alpha\beta}\subseteq\acfrg{\beta}$ one has $D\in\acfrg{\beta}$ and from $D=D\cap r(\alpha\beta)$ it can be concluded that $D\cap C=(D\cap r(\alpha\beta))\cap C=\emptyset$, and thus $D\notin\ft$, again by Proposition \ref{prop:ultrafilter.intersection} and using that $\ft$ is an ultrafilter.
	That means that there must be an index $\lambda'$ such that $\lambda\geq\lambda'$ implies $D\notin\ft_{\lambda}$, ensuring the existence of an element $C_{\lambda}\in\ft_{\lambda}$ with $D\cap C_{\lambda}=\emptyset$ for each such $\lambda$; the element $C_{\lambda}\cap r(\alpha\beta)\in g_{(\alpha)\beta}(\ft_{\lambda})$ satisfies \[ D\cap(C_{\lambda}\cap r(\alpha\beta))=\emptyset, \] from where it can be established that $D\notin g_{(\alpha)\beta}(\ft_{\lambda})$ for all $\lambda\geq\lambda'$. This concludes the proof that $g_{(\alpha)\beta}(\ft_{\lambda})$ converges to $g_{(\alpha)\beta}(\ft)$, and that $g_{(\alpha)\beta}$ is therefore a continuous map.
	
	As for (iii), suppose $\ft\in X_{(\alpha)\beta}$ and $\{\ft_\lambda\}_\lambda\subseteq X_\beta$ is a net that converges to $\ft$. Since $r(\alpha\beta)\in\ft$, the pointwise convergence says there is an index $\lambda_0$ such that $\lambda\geq\lambda_0$ implies $r(\alpha\beta)\in\ft_\lambda$, and thus $\ft_\lambda\in X_{(\alpha)\beta}$, whence $X_{(\alpha)\beta}$ is open.
\end{proof}

Let us proceed with the problem of describing new filters by cutting labelled paths.

Consider composable labelled paths $\alpha\in\awplus$ and $\beta\in\awstar$, and an ultrafilter $\ft\in X_{\alpha\beta}$. Note that $\ft\subseteq\acfrg{\alpha\beta}\subseteq\acfrg{\beta}$, but $\ft$ may not be a filter in $\acfrg{\beta}$, for $\acfrg{\beta}$ may contain elements above a given element of $\ft$ that are not in $\acfrg{\alpha\beta}$. If we add to $\ft$ these elements, however, the resulting set is an ultrafilter in $\acfrg{\beta}$, as the following result shows.

\begin{proposition}
	\label{prop:h-class.definition.enabler}
	Let $\alpha\in\awplus$ and $\beta\in\awstar$ be such that $\alpha\beta\in\awplus$, and $\ft\in X_{\alpha\beta}$. Then $\usetr{\ft}{\acfrg{\beta}}\in X_{(\alpha)\beta}$.
\end{proposition}
\begin{proof}
	Suppose $C\in\acfrg{\beta}$ satisfies $C\cap B\neq\emptyset$ for all $B\in\usetr{\ft}{\acfrg{\beta}}$;  in particular, $C\cap B\neq\emptyset$ for all $B\in\ft$. Note that $r(\alpha\beta)\in\ft$ (since $\ft\in X_{\alpha\beta}$), so for any given $B\in\ft$ one has $B\cap r(\alpha\beta)=B$,  which in turn gives, for the element $(C\cap r(\alpha\beta))\in\lset{r(\alpha\beta)}=\acfrg{\alpha\beta}$, \[ (C\cap r(\alpha\beta))\cap B= C\cap(r(\alpha\beta)\cap B)=C\cap B\neq\emptyset. \] $\ft$ is an ultrafilter, so $C\cap r(\alpha\beta)\in\ft$ and thus $C\in\usetr{(C\cap r(\alpha\beta))}{\acfrg{\beta}}\subseteq\usetr{\ft}{\acfrg{\beta}}$. Proposition \ref{prop:ultrafilter.intersection} now ensures that $\usetr{\ft}{\acfrg{\beta}}\in X_\beta$, and clearly $r(\alpha\beta)\in\usetr{\ft}{\acfrg{\beta}}$, proving $\usetr{\ft}{\acfrg{\beta}}\in X_{(\alpha)\beta}$ as desired.
\end{proof}

Proposition \ref{prop:h-class.definition.enabler} gives rise to a function $\newf{h_{[\alpha]\beta}}{X_{\alpha\beta}}{X_{(\alpha)\beta}}$ that associates to each ultrafilter $\ft\in X_{\alpha\beta}$, the ultrafilter in $\acfrg{\beta}$ given by \[ h_{[\alpha]\beta}(\ft)=\usetr{\ft}{\acfrg{\beta}}. \] The following lemmas describe some of the properties of these maps.

\begin{lemma}
	\label{lemma:h.class.properties}
	Let $\alpha\in\awplus$ and $\beta,\gamma\in\awstar$ with $\alpha\beta\gamma\in\awplus$ be given. Then,
	\begin{enumerate}[(i)]
		\item\label{item:h-class.composition.rule}
		$h_{[\beta]\gamma}\circ h_{[\alpha]\beta\gamma}=h_{[\alpha\beta]\gamma}$.
		
		\item\label{item:f-class.and.h-class.compatible}
		The following diagram is commutative:
		\[\xymatrix{
			X_{\alpha\beta\gamma} \ar[r]^{h_{[\alpha]\beta\gamma}} \ar[d]_{f_{\alpha\beta[\gamma]}} & X_{(\alpha)\beta\gamma} \ar[d]^{f_{\beta[\gamma]}} \\
			X_{\alpha\beta} \ar[r]_{h_{[\alpha]\beta}}       & X_{(\alpha)\beta}. }
		\]
	\end{enumerate}
\end{lemma}
\begin{proof}
	For a given $\ft\in X_{\alpha\beta\gamma}$, since $\ft\subseteq \usetr{\ft}{\acfrg{\beta\gamma}}$ one obtains \[h_{[\alpha\beta]\gamma}(\ft)=\usetr{\ft}{\acfrg{\gamma}}\subseteq \usetr{\left(\usetr{\ft}{\acfrg{\beta\gamma}}\right)}{\acfrg{\gamma}}=h_{[\beta]\gamma}\circ h_{[\alpha]\beta\gamma}(\ft), \] which implies these are equal since they are both ultrafilters in $\acfrg{\gamma}$. This proves (i). As for (ii), first observe that Lemma \ref{lemma:f-class.preserves.Xalphabetas}.(\ref{item:f-class.preserves.parentheses}) says the arrow on the right of the diagram makes sense. If $\ft\in X_{\alpha\beta\gamma}$ then \[ h_{[\alpha]\beta}\circ f_{\alpha\beta[\gamma]}(\ft)=\{D\in\acfrg{\beta}\ |\ \exists\, C\in\acfrg{\alpha\beta}\ \text{with}\ r(C,\gamma)\in\ft\ \text{and}\  C\subseteq D \}. \] For a given $D\in h_{[\alpha]\beta}\circ f_{\alpha\beta[\gamma]}(\ft)$, if $C\in\acfrg{\alpha\beta}$ is as above then $r(C,\gamma)\subseteq r(D,\gamma)\in \acfrg{\beta\gamma}$, so that \[ r(D,\gamma)\in\usetr{r(C,\gamma)}{\acfrg{\beta\gamma}}\subseteq\usetr{\ft}{\acfrg{\beta\gamma}}=h_{[\alpha]\beta\gamma}(\ft), \] and therefore $D\in f_{\beta[\gamma]}\circ h_{[\alpha]\beta\gamma}(\ft)$. The result now follows since both $f_{\beta[\gamma]}\circ h_{[\alpha]\beta\gamma}(\ft)$ and $h_{[\alpha]\beta}\circ f_{\alpha\beta[\gamma]}(\ft)$ are ultrafilters in $\acfrg{\beta}$.
\end{proof}

\begin{lemma}
	\label{lemma:h-class.properties2}
	Let $\alpha\in\awplus$ and $\beta\in\awstar$ be such that $\alpha\beta\in\awplus$. Then,
	\begin{enumerate}[(i)]
		\item\label{item:h-class.g-class.inverses} The functions $\newf{h_{[\alpha]\beta}}{X_{\alpha\beta}}{X_{(\alpha)\beta}}$ and $\newf{g_{(\alpha)\beta}}{X_{(\alpha)\beta}}{X_{\alpha\beta}}$ are mutual inverses.
		
		\item\label{item:h-class.preserves.sinks} $h_{[\alpha]\beta}(X_{\alpha\beta}^{sink})\subseteq X_\beta^{sink}.$
		
		\item\label{item:h-class.continuous}
		$\newf{h_{[\alpha]\beta}}{X_{\alpha\beta}}{X_{(\alpha)\beta}}$ is continuous.
	\end{enumerate}
	
\end{lemma}
\begin{proof}
	To prove (i), if $\ft\in X_{(\alpha)\beta}$ then $h_{[\alpha]\beta}\circ g_{(\alpha)\beta}(\ft)$ is an ultrafilter in $\acfrg{\beta}$; also, for each $C\in\ft$ one has $C\in\usetr{C\cap r(\alpha\beta)}{\acfrg{\beta}}\subseteq h_{[\alpha]\beta}\circ g_{(\alpha)\beta}(\ft)$, which shows that $\ft\subseteq h_{[\alpha]\beta}\circ g_{(\alpha)\beta}(\ft)$, hence they are equal. On the other hand if $\ft\in X_{\alpha\beta}$, then \[ g_{(\alpha)\beta}\circ h_{[\alpha]\beta}(\ft)= \{D\cap r(\alpha\beta)\ |\ D\in\acfrg{\beta}\ \text{such that}\ \exists\, C\in\ft\ \text{with}\ C\subseteq D \}. \]
	Given $C\in\ft$, from $\acfrg{\alpha\beta}\subseteq\acfrg{\beta}$ it can be seen that $C=C\cap r(\alpha\beta)$ and thus $C\in g_{(\alpha)\beta}\circ h_{[\alpha]\beta}(\ft)$, establishing $\ft\subseteq g_{(\alpha)\beta}\circ h_{[\alpha]\beta}(\ft)$ and again equality follows.
	
	Now, let us prove (ii): suppose $\ft\in X_{\alpha\beta}^{sink}$; then, for each $a\in\alf$ there exists $D\in\ft$ such that $r(D,a)=\emptyset$. Since $D\in\ft\subseteq\usetr{\ft}{\acfrg{\beta}}=h_{[\alpha]\beta}(\ft)$ it can be concluded that for each $a\in\alf$ there exists $D\in h_{[\alpha]\beta}(\ft)$ such that $r(D,a)=\emptyset$, which means $h_{[\alpha]\beta}(\ft)\in X_{\beta}^{sink}$, as desired.
	
	As for (iii), consider a net $\{\ft_\lambda\}_\lambda\subseteq X_{\alpha\beta}$ that converges to $\ft\in X_{\alpha\beta}$, and let $D\in\acfrg{\beta}$ be arbitrary. If $D\in h_{[\alpha]\beta}(\ft)$ then there exists $C\in\ft$ such that $C\subseteq D$. The convergence of the $\ft_\lambda$ say then that there is an index $\lambda_0$ such that $\lambda\geq\lambda_0$ implies $C\in\ft_\lambda$, and hence \[ D\in\usetr{C}{\acfrg{\beta}}\subseteq\usetr{\ft_\lambda}{\acfrg{\beta}}=h_{[\alpha]\beta}(\ft_\lambda). \] On the other hand, if $D\in\acfrg{\beta}\backslash h_{[\alpha]\beta}(\ft)$ then there exists $C\in h_{[\alpha]\beta}(\ft)$ such that $C\cap D=\emptyset$, by Proposition \ref{prop:ultrafilter.intersection}. What has been just shown above says there exists $\lambda_0$ such that $\lambda\geq\lambda_0$ implies $C\in h_{[\alpha]\beta}(\ft_\lambda)$, and again Proposition \ref{prop:ultrafilter.intersection} can be used to ensure $\lambda\geq\lambda_0$ implies $D\notin h_{[\alpha]\beta}(\ft_\lambda)$; it then follows that the net $\{h_{[\alpha]\beta}(\ft_\lambda)\}_\lambda$ converges to $h_{[\alpha]\beta}(\ft)$, hence $h_{[\alpha]\beta}$ is continuous.
\end{proof}

The functions above can be used to produce tools for working with filters in $E(S)$. Let us begin with a function for gluing labelled paths to a filter $\xi$ in $E(S)$. Given labelled paths $\alpha\in\awplus$ and $\beta\in\awleinf$ such that $\alpha\beta\in\awleinf$, the problem is to construct a complete family of ultrafilters for $\alpha\beta$ starting from a complete family of ultrafilters for $\beta$. Roughly, one begins with the complete family for $\beta$, say $\{\ft_n\}_n$ with $0\leq n\leq |\beta|$, cuts each ultrafilter $\ft_n$ by the range of $\alpha\beta_{1,n}$, and adds new ultrafilters at the beginning of the family, one for each $\alpha_{1,i}$, $0\leq i\leq|\alpha|$, to obtain a complete family $\{\ftg{J}_i\}_i$ for $\alpha\beta$ (remembering that, if the resulting family is to be complete, we have no real choice as to which filters to add at the beginning of the family, from what was seen in Section \ref{section:preliminaries}).

More precisely, consider first the case $\beta=\eword$. In this case the complete family $\{\ft_n\}_n$ for $\eword$ consists of a single filter $\ft_0\subseteq\acf$. Define \[ \ftg{J}_{|\alpha|}=\{C\cap r(\alpha)\ |\ C\in\ft_0\}=g_{(\alpha)\eword}(\ft_0) \] and, for for $0\leq i< |\alpha|$, set \[ \ftg{J}_i=\{D\in\acfrg{\alpha_{1,i}}\ |\ r(D,\alpha_{i+1,|\alpha|})\in\ftg{J}_{|\alpha|}\}=f_{\alpha_{1,i}[\alpha_{i+1,|\alpha|}]}(\ftg{J}_{|\alpha|}). \]
Under a suitable condition on the filter $\ft_0$ (Proposition \ref{prop:G-class.definition.enabler}.(\ref{item:G-class.definition.enabler.eword}) below), the family $\{\ftg{J}_i\}_i$ obtained is a complete family for $\alpha=\alpha\eword$.

Now, for the case where $\beta\neq\eword$: the ultrafilter $\ft_1\in X_{\beta_1}$ is translated $|\alpha|$ units to create the ultrafilter $\ftg{J}_{|\alpha|+1}\in X_{\alpha\beta_1}$, given by \[ \ftg{J}_{|\alpha|+1} = \{C\cap r(\alpha\beta_1)\ |\ C\in\ft_1\}=g_{(\alpha)\beta_1}(\ft_1). \]

More generally, for $1\leq n\leq|\beta|$ (or $n<|\beta|$ if $\beta$ is infinite) one defines
\[ \ftg{J}_{|\alpha|+n} = \{C\cap r(\alpha\beta_{1,n})\ |\ C\in\ft_n\}=g_{(\alpha)\beta_{1,n}}(\ft_n). \]

\begin{remark}
	\label{remark:G-class.ft1.enough}
	Note that, in order for this construction to work, one must have $\ft_n\in X_{(\alpha)\beta_{1,n}}$ for all $n$, which follows from requiring $\ft_1\in X_{(\alpha)\beta_1}$ and using the completeness of $\{\ft_n\}_n$ for $\beta$. Indeed, if $\ft_n\in X_{(\alpha)\beta_n}$ then $r(\alpha\beta_{1,n})\in\ft_n$, and so \[ r(\alpha\beta_{1,n+1})=r(r(\alpha\beta_{1,n}),\beta_{n+1})\in \ft_{n+1}, \] that is, $\ft_{n+1}\in X_{(\alpha)\beta_{1,n+1}}$.
\end{remark}
Also, for $0\leq i\leq|\alpha|$, define \dmal{\ftg{J}_i&=f_{\alpha_{1,i}[\alpha_{i+1,|\alpha|}\beta_1]}(\ftg{J}_{|\alpha|+1}) \\
	&=\{D\in\acfrg{\alpha_{1,i}}\ |\ r(D,\alpha_{i+1,|\alpha|}\beta_1)\in\ftg{J}_{|\alpha|+1}\}.}

\begin{remark}\label{remark:G-class.ft0.enough.if.not.empty}
	If it is the case that $\ft_0\neq\emptyset$ in the complete family $\{\ft_n\}_n$, then necessarily one has $\ftg{J}_{|\alpha|}=g_{(\alpha)\eword}(\ft_0)$. This is true because, as mentioned above, $\ftg{J}_{|\alpha|}=f_{\alpha_{1,|\alpha|}[\alpha_{|\alpha|+1,|\alpha|}\beta_1]}(\ftg{J}_{|\alpha|+1})=f_{\alpha[\beta_1]}(\ftg{J}_{|\alpha|+1})$, and therefore
	\dmal{\ftg{J}_{|\alpha|}=f_{\alpha[\beta_1]}(\ftg{J}_{|\alpha|+1})&=f_{\alpha\eword[\beta_1]}\circ g_{(\alpha)\eword\beta_1}(\ft_1) \\ &= g_{(\alpha)\eword}\circ f_{\eword[\beta_1]}(\ft_1)=g_{(\alpha)\eword}(\ft_0),} using Lemma \ref{lemma:g-class.properties} (\ref{item:f-class.and.g-class.compatible}).
\end{remark}

\begin{proposition}
	\label{prop:G-class.definition.enabler}
	Suppose the labelled space $\lspace$ is weakly left-resolving, and let $\alpha\in\awplus$.
	\begin{enumerate}[(i)]
		\item\label{item:G-class.definition.enabler.eword} If $\ft_0\in X_{(\alpha)\eword}$, then $\{\ftg{J}_i\}_i$, $0\leq i\leq |\alpha|$ as above is a complete family of ultrafilters for $\alpha$.
		\item\label{item:G-class.definition.enabler.not.eword} If $\beta\in\awleinf\backslash\{\eword\}$ is such that $\alpha\beta\in\awleinf$ and $\{\ft_n\}_n$ is a complete family of ultrafilters for $\beta$ such that $\ft_1\in X_{(\alpha)\beta_1}$, then $\{\ftg{J}_i\}_i$ as above is a complete family of ultrafilters for $\alpha\beta$.
	\end{enumerate} 
\end{proposition}
\begin{proof}
	The proofs of both items are similar. We prove (ii): for $0\leq i<|\alpha|$, \dmal{\ftg{J}_{i}&= f_{\alpha_{1,i}[\alpha_{i+1,|\alpha|}\beta_1]}(\ftg{J}_{|\alpha|+1})\\&=f_{\alpha_{1,i}[\alpha_{i+1}]}\circ f_{\alpha_{1,i+1}[\alpha_{i+2,|\alpha|}\beta_1]}(\ftg{J}_{|\alpha|+1})=f_{\alpha_{1,i}[\alpha_{i+1}]}(\ftg{J}_{i+1}),} and additionally $\ftg{J}_{|\alpha|}=f_{\alpha[\beta_1]}(\ftg{J}_{|\alpha|+1})$ by definition. On the other hand for $i>|\alpha|$, say $i=|\alpha|+n$ for $n\geq 1$, the definitions above coupled with Lemma \ref{lemma:g-class.properties}.(\ref{item:f-class.and.g-class.compatible}) and the fact the family $\{\ft_n\}_n$ is complete for $\beta$ give \dmal{\ftg{J}_i&=g_{(\alpha)\beta_{1,n}}(\ft_n)=g_{(\alpha)\beta_{1,n}}(f_{\beta_{1,n}[\beta_{n+1}]}(\ft_{n+1}))\\
		&=f_{\alpha\beta_{1,n}[\beta_{n+1}]}\circ g_{(\alpha)\beta_{1,n+1}}(\ft_{n+1})=f_{\alpha\beta_{1,n}[\beta_{n+1}]}(\ftg{J}_{i+1}),} and thus the family $\{\ftg{J}_i\}_i$ is complete for $\alpha\beta$, as claimed.
\end{proof}

Continuing the discussion above, let us now concentrate on the tight filters in $E(S)$: for $\beta\in\awleinf$ let $\ftightw{(\alpha)\beta}$ denote the subset of $\ftightw{\beta}$ given by \[ \ftightw{(\alpha)\beta}=\{\xi\in\ftightw{\beta}\ |\ \xi_0\in X_{(\alpha)\eword}\}. \]

Observe that, for $\beta\neq\eword$, $\xi_0\in X_{(\alpha)\eword}$ is equivalent to $\xi_1\in X_{(\alpha)\beta_1}$, and $r(\alpha)\in\xi_0$ is equivalent to $r(\alpha\beta_1)\in\xi_1$.

For a given tight filter $\xi\in\ftightw{(\alpha)\beta}$, Proposition \ref{prop:G-class.definition.enabler} says the complete family $\{\xi_n\}_n$ for $\beta$ gives rise to a complete family $\{\ftg{J}_i\}_i$ of ultrafilters for $\alpha\beta$, and therefore to a filter $\eta\in\filt_{\alpha\beta}$ associated with this family. The purpose of the next theorem is to show this resulting filter is also tight.

\begin{theorem}
	\label{thm:G-class.preserves.tight}
	Suppose the labelled space $\lspace$ is weakly left-resolving, that $\acf$ is closed under relative complements, and let $\alpha\in\awplus$ and $\beta\in\awleinf$ be such that $\alpha\beta\in\awleinf$. If $\xi\in\ftightw{(\alpha)\beta}$ and $\{\ftg{J}_i \}_i$ is the complete family for $\alpha\beta$ constructed as above from $\{\xi_n\}_n$, then the filter $\eta\in\filtw{\alpha\beta}$ associated with $\{\ftg{J}_i\}_i$ is tight.
\end{theorem}
\begin{proof}
	Suppose first that $\beta\in\awinf$. In this case $\{\xi_n\}_n$ is a complete family of ultrafilters for $\beta$ (by Theorems \ref{thm:tight.filters.in.es}.(\ref{item:tight.filters.in.es.infinite.type}) and \ref{thm:ultrafilters}.(\ref{item:ultrafilters.infinite.type.special.case})'), and $\xi_1\in X_{(\alpha)\beta_1}$ since $\xi\in\ftightw{(\alpha)\beta}$. Proposition \ref{prop:G-class.definition.enabler} then says the family $\{\ftg{J}_i\}_i$ constructed from $\{\xi_n\}_n$ is a complete family of ultrafilters for $\alpha\beta$, and thus the filter $\eta\in\filtw{\alpha\beta}$ associated with this family is tight, again by Theorems \ref{thm:tight.filters.in.es}.(\ref{item:tight.filters.in.es.infinite.type}) and \ref{thm:ultrafilters}.(\ref{item:ultrafilters.infinite.type.special.case})'.
	
	Next, suppose $\beta\in\awstar$; from Proposition \ref{prop:tight.filters.of.finite.type}, $\xi$ may be of one of two kinds of tight filters, and we consider each in turn: for the first kind, there exists a net $\{\ft_\lambda\}_\lambda\subseteq X_{\beta}^{sink}$ that converges to $\xi_{|\beta|}\in X_{(\alpha)\beta}$. The set $X_{(\alpha)\beta}$ is open in $X_\beta$ by Lemma \ref{lemma:g-class.properties2}.(\ref{item:Xalphabeta.open}), so there is an index $\lambda_0$ such that $\lambda\geq\lambda_0$ implies $\ft_\lambda\in X_{(\alpha)\beta}$. Now $\{g_{(\alpha)\beta}(\ft_\lambda)\}_{\lambda\geq\lambda_0}$ is a net in $X_{\alpha\beta}^{sink}$ due to Lemma \ref{lemma:g-class.properties2}.(\ref{item:g-class.preserves.sinks}), and it converges to $\ftg{J}_{|\alpha\beta|}=g_{(\alpha)\beta}(\xi_{|\beta|})$, by the continuity of $g_{(\alpha)\beta}$ established with Lemma \ref{lemma:g-class.properties2}.(\ref{item:g-class.continuous}), whence $\eta$ is tight.
	
	For the second kind, there is a net $\{(t_\lambda,\ft_\lambda)\}_\lambda$ with $t_\lambda\in\alf$ and $\ft_\lambda\in X_{\beta t_\lambda}$ for all $\lambda$ such that $\{t_\lambda\}_\lambda$ converges to infinity in $\alf$ and $\{f_{\beta[t_\lambda]}(\ft_\lambda)\}_\lambda$ converges to $\xi_{|\beta|}$ in $X_\beta$. Again there is an index $\lambda_0$ such that $\lambda\geq\lambda_0$ implies $f_{\beta[t_\lambda]}(\ft_\lambda)\in X_{(\alpha)\beta}$, and for these $\lambda$ one must then have $\ft_\lambda\in X_{(\alpha)\beta t_\lambda}$, by Lemma \ref{lemma:f-class.preserves.Xalphabetas}. The net $\{g_{(\alpha)\beta t_\lambda}(\ft_\lambda)\}_{\lambda\geq\lambda_0}$ satisfies, as given by Lemma \ref{lemma:g-class.properties}.(\ref{item:f-class.and.g-class.compatible}), \[ f_{\alpha\beta[t_\lambda]}(g_{(\alpha)\beta t_\lambda}(\ft_\lambda))= g_{(\alpha)\beta}(f_{\beta[t_{\lambda}]}(\ft_\lambda)),\] and this converges to $g_{(\alpha)\beta}(\xi_{|\beta|})=\ftg{J}_{|\alpha\beta|}$, again showing that $\eta$ is tight, for $\{t_\lambda\}_{\lambda\geq\lambda_0}$ still converges to infinity in $\alf$.
	
\end{proof}

Under the hypotheses of Theorem \ref{thm:G-class.preserves.tight} above, it is then possible to define a function \[\newf{G_{(\alpha)\beta}}{\ftightw{(\alpha)\beta}}{\ftightw{\alpha\beta}}\] taking a tight filter $\xi\in\ftightw{(\alpha)\beta}$ to the tight filter $\eta\in\ftightw{\alpha\beta}$ given by the theorem. Also, for $\alpha=\eword$ define $\ftightw{(\eword)\beta}=\ftightw{\beta}$ and let $G_{(\eword)\beta}$ be the identity function on $\ftightw{\beta}$.

\begin{lemma}
	\label{lemma:G-class.properties}
	Suppose the labelled space $\lspace$ is weakly left-resolving, that $\acf$ is closed under relative complements, and let $\alpha,\beta\in\awplus$ and $\gamma\in\awleinf$ with $\alpha\beta\gamma\in\awleinf$ be given. Then,
	\begin{enumerate}[(i)]
		\item\label{item:G-class.domain.containment}
		$\ftightw{(\alpha\beta)\gamma}\subseteq\ftightw{(\beta)\gamma}$;
		\item\label{item:G-class.composition.rule}
		$G_{(\alpha\beta)\gamma}=G_{(\alpha)\beta\gamma}\circ G_{(\beta)\gamma}$.
	\end{enumerate}
\end{lemma}
\begin{proof}
	For (i), Lemma \ref{lemma:g-class.properties}.(\ref{item:g-class.domain.containment}) states $X_{(\alpha\beta)\gamma_1}\subseteq X_{(\beta)\gamma_1}$, at once giving $\ftightw{(\alpha\beta)\gamma}\subseteq\ftightw{(\beta)\gamma}$. As for (ii), consider an arbitrary $\xi\in\ftightw{_{(\alpha\beta)\gamma}}$, and let $\eta=G_{(\alpha\beta)\gamma}(\xi)$, $\sigma=G_{(\beta)\gamma}(\xi)$ and $\rho=G_{(\alpha)\beta\gamma}(\sigma)$; $\eta$ and $\rho$ are both tight filters in $\acfrg{\alpha\beta\gamma}$. For any given $n\geq 1$, note that \dmal{ \eta_{|\alpha\beta|+n}&=g_{(\alpha\beta)\gamma_{1,n}}(\xi_n)=g_{(\alpha)\beta\gamma_{1,n}}\circ g_{(\beta)\gamma_{1,n}}(\xi_n)\\
		&=g_{(\alpha)\beta\gamma_{1,n}}(\sigma_{|\beta|+n})=\rho_{|\alpha|+|\beta|+n}=\rho_{|\alpha\beta|+n},
	}
	and from this one sees that $\eta$ and $\rho$ are associated with the same complete family of ultrafilters (remember, each ultrafilter in a complete family determines uniquely the ones preceding it). This means $\eta=\rho$, and (ii) follows.
\end{proof}

Next is a function for removing labelled paths from a filter $\xi\in E(S)$. The problem now is to construct, from a complete family of ultrafilters for $\alpha\beta$, a new complete family of ultrafilters for $\beta$. This is achieved with the following result.

\begin{lemma}
	\label{prop:H-class.definition.enabler}
	Let $\alpha\in\awplus$ and $\beta\in\awleinf$ be such that $\alpha\beta\in\awleinf$. If $\{\ft_n\}_n$ is a complete family of ultrafilters for $\alpha\beta$, then \[\{h_{[\alpha]\beta_{1,n}}(\ft_{|\alpha|+n})\}_n\] is a complete family of ultrafilters for $\beta$.
\end{lemma}
\begin{proof}
	If $\beta=\eword$ then $h_{[\alpha]\eword}(\ft_{|\alpha|})\in X_{(\alpha)\eword}\subseteq X_{\eword}$, so there is nothing to do. 
	Suppose then $\beta\neq\eword$; for any given $n\geq 1$, from Lemma \ref{lemma:h.class.properties}.(\ref{item:f-class.and.h-class.compatible}) one has \dmal{h_{[\alpha]\beta_{1,n}}(\ft_{|\alpha|+n})&=h_{[\alpha]\beta_{1,n}}(f_{\alpha\beta_{1,n}[\beta_{n+1}]}(\ft_{|\alpha|+n+1}))\\
		&=f_{\beta_{1,n}[\beta_{n+1}]}\circ h_{[\alpha]\beta_{1,n+1}}(\ft_{|\alpha|+n+1}),} which is precisely the desired completeness.
\end{proof}

\begin{theorem}
	\label{thm:H-class.preserves.tight}
	Suppose that the labelled space $\lspace$ is weakly left-resolving, that $\acf$ is closed under relative complements, and let $\alpha\in\awplus$ and $\beta\in\awleinf$ be such that $\alpha\beta\in\awleinf$. If $\xi\in\ftightw{\alpha\beta}$, then the filter $\eta\in\filtw{\beta}$ associated with the complete family $\{h_{[\alpha]\beta_{1,n}}(\xi_{n+|\alpha|})\}_{n}$ for $\beta$ is a tight filter.
\end{theorem}
\begin{proof}
	Again the proof is split into cases, using Theorem \ref{thm:tight.filters.in.es} and Proposition \ref{prop:tight.filters.of.finite.type}. Begin with the case $\beta\in\awinf$, so that $\xi$ is an ultrafilter of infinite type; Lemma \ref{prop:H-class.definition.enabler} says the family $\{h_{[\alpha]\beta_{1,n}}(\xi_{n+|\alpha|})\}_{n}$ of ultrafilters is indeed complete for $\beta$, hence the filter $\eta\in\filtw{\beta}$ associated with it is an ultrafilter and thus tight.
	
	Next, consider $\beta\in\awstar$, and suppose there exists a net $\{\ft_\lambda\}_\lambda\subseteq X_{\alpha\beta}^{sink}$ that converges to $\xi_{|\alpha\beta|}\in X_{\alpha\beta}$. The map $h_{[\alpha]\beta}$ is continuous and preserves sinks by Lemma \ref{lemma:h-class.properties2}, meaning that $\{h_{[\alpha]\beta}(\ft_\lambda)\}_\lambda$ is a net in $X_\beta^{sink}$ that converges to $h_{[\alpha]\beta}(\xi_{|\alpha\beta|})$, implying in turn that $\eta$ is tight.
	
	Finally, suppose that $\beta\in\awstar$ and that there exists a net $\{(t_\lambda,\ft_\lambda)\}_\lambda$ with $t_\lambda\in\alf$ and $\ft_\lambda\in X_{\alpha\beta t_\lambda}$ for all $\lambda$ such that $\{t_\lambda\}_\lambda$ converges to infinity in $\alf$ and $\{f_{\alpha\beta[t_\lambda]}(\ft_\lambda)\}_\lambda$ converges to $\xi_{|\alpha\beta|}$ in $X_{\alpha\beta}$. The commutativity of the diagram in Lemma \ref{lemma:h.class.properties}.(\ref{item:f-class.and.h-class.compatible}) ensures that		
	\[ h_{[\alpha]\beta}(f_{\alpha\beta[t_\lambda]}(\ft_\lambda))=f_{\beta[t_\lambda]}(h_{[\alpha]\beta t_\lambda}(\ft_\lambda)), \] and from here it can be seen, using the continuity of $h_{[\alpha]\beta}$, that the net $\{(t_\lambda,h_{[\alpha]\beta t_\lambda}(\ft_\lambda))\}_\lambda$ is such that $\{t_\lambda\}_\lambda$ converges to infinity in $\alf$ and 	 the net $\{f_{\beta[t_\lambda]}(h_{[\alpha]\beta t_\lambda}(\ft_\lambda))\}_\lambda$ in $X_\beta$ converges to $h_{[\alpha]\beta}(\xi_{|\alpha\beta|})$, whence $\eta$ is tight.
\end{proof}

Under the hypotheses of Theorem \ref{thm:H-class.preserves.tight} one can therefore define a function \[\newf{H_{[\alpha]\beta}}{\ftightw{\alpha\beta}}{\ftightw{(\alpha)\beta}}\] by $H_{[\alpha]\beta}(\xi^{\alpha\beta})=\eta^\beta$ where, for all $n$ with $0\leq n\leq |\beta|$, \[ \eta^\beta_n=h_{[\alpha]\beta_{1,n}}(\xi_{n+|\alpha|})\in X_{(\alpha)\beta_{1,n}}. \] For $\alpha=\eword$ define $H_{[\eword]\beta}$ to be the identity function over $\ftightw{\beta}$.

\begin{lemma}
	\label{lemma:H-class.composition.rule}
	Suppose the labelled space $\lspace$ is weakly left-resolving, that $\acf$ is closed under relative complements, and let $\alpha,\beta\in\awplus$ and $\gamma\in\awleinf$ with $\alpha\beta\gamma\in\awleinf$ be given. Then \[ H_{[\beta]\gamma}\circ H_{[\alpha]\beta\gamma}=H_{[\alpha\beta]\gamma}. \]
\end{lemma}
\begin{proof}
	Immediate from Lemma \ref{lemma:h.class.properties}.(\ref{item:h-class.composition.rule}).
\end{proof}

\begin{theorem}
	\label{thm:H-class.G-class.inverses}
	Suppose the labelled space $\lspace$ is weakly left-resolving, that $\acf$ is closed under relative complements, and let $\alpha\in\awplus$ and $\beta\in\awleinf$ be such that $\alpha\beta\in\awleinf$. Then $H_{[\alpha]\beta}\circ G_{(\alpha)\beta}$ and $G_{(\alpha\beta)}\circ H_{[\alpha]\beta}$ are the identity maps over $\ftightw{(\alpha)\beta}$ and $\ftightw{\alpha\beta}$, respectively.
\end{theorem}
\begin{proof}
	We consider the case $\beta\neq\eword$, as the case $\beta=\eword$ is analogous. Let $\xi\in\ftightw{(\alpha)\beta}$ be arbitrary, and denote $\sigma^{\alpha\beta}=G_{(\alpha)\beta}(\xi)$, $\eta^\beta=H_{[\alpha]\beta}(\sigma)$. For any given integer $n$ with $1\leq n\leq |\beta|$ (note that $\xi_0$ may be empty), \[ \eta_n=h_{[\alpha]\beta_{1,n}}(\sigma_{n+|\alpha|})=h_{[\alpha]\beta_{1,n}}(g_{(\alpha)\beta_{1,n}}(\xi_n))=\xi_n, \] as a consequence of Lemma \ref{lemma:h-class.properties2}.(\ref{item:h-class.g-class.inverses}), giving immediately $\eta=\xi$, from where it can be concluded that $H_{[\alpha]\beta}\circ G_{(\alpha)\beta}$ is the identity map over $\ftightw{(\alpha)\beta}$.
	
	On the other hand, begin with an arbitrary $\xi\in\ftightw{\alpha\beta}$ and let $\eta^\beta=H_{[\alpha]\beta}(\xi)$, $\sigma^{\alpha\beta}=G_{(\alpha)\beta}(\eta)$. If $n$ is an integer with $1\leq n\leq |\beta|$ then \[ \sigma_{n+|\alpha|}=g_{(\alpha)\beta_{1,n}}(\eta_n)=g_{(\alpha)\beta_{1,n}}(h_{[\alpha]\beta_{1,n}}(\xi_{n+|\alpha|}))=\xi_{n+|\alpha|}, \] again from Lemma \ref{lemma:h-class.properties2}.(\ref{item:h-class.g-class.inverses}), giving $\sigma_{n+|\alpha|}=\xi_{n+|\alpha|}$ and thus $\sigma=\xi$ (for $\sigma_{1+|\alpha|}=\xi_{1+|\alpha|}$ implies $\sigma_i=\xi_i$ for all $i\leq |\alpha|$), that is, $G_{(\alpha\beta)}\circ H_{[\alpha]\beta}$ is the identity map over $\ftightw{\alpha\beta}$.
\end{proof}

\section{Representations}
\label{section:representations}

For this section, let $\lspace$ be a weakly left-resolving labelled space such that $\acf$ is closed under relative complements, and let $S$ be its inverse semigroup, as in Subsection \ref{subsection:inverse.semigroup}. Our goal is to build a representation of $C^*\lspace$ that shows that an element $s_{\alpha}p_A s_{\beta}^*\in C^*\lspace$ is non-zero whenever $(\alpha,A,\beta)\in S\setminus\{0\}$. To achieve this goal, we make use of the functions $G$ and $H$ defined in Section \ref{section:filter.surgery}.

Let $\hs=\ell^2(\ftight)$. Throughout this section we make an abuse of notation in that $\xi$ represents both an element of $\ftight$ and of $\hs$. For each $A\in\acf$, define
\[\hs_A=\overline{\text{span}}\left\{\xi\in\ftight\,|\,(\eword,A,\eword)\in\xi\right\}=\overline{\text{span}}\left\{\xi\in\ftight\,|\,A\in\xi_0\right\}.\]
Also, for $a\in\alf$, define
\[\hs_a=\overline{\text{span}}\left\{\xia\in\ftight\,|\,\alpha_1=a\right\}.\]

For a given $A\in\acf$, let $P_A\in \la{B}(\hs)$ be the orthogonal projection onto $\hs_A$; additionally, for $a\in\alf$, we can use Theorem \ref{thm:H-class.G-class.inverses} to define a partial isometry $S_a\in\la{B}(\hs)$ with initial space $\hs_{r(a)}$ and final space $\hs_a$, given by $S_a(\xi^{\beta})=G_{(a)\beta}(\xi)$, for all $\xi=\xi^{\beta}\in\hs_{r(a)}$. It follows that $S_a^*$ is given by $S_a^*(\xi^{a\beta})=H_{[a]\beta}(\xi)$, for all $\xi=\xi^{a\beta}\in\hs_a$. If $\alpha=\alpha_1\cdots\alpha_n\in\awplus$, set $S_{\alpha}=S_{\alpha_1}\cdots S_{\alpha_n}$.

\begin{proposition}
	The family $\{P_A,S_a\}$ satisfies the relations defining $C^*\lspace$ in Definition \ref{def:c*-algebra.labelled.space.v3}.
\end{proposition}

\begin{proof}
	Note that $\hs_{\emptyset}=\{0\}$ since $(\eword,\emptyset,\eword)\notin S$ and therefore $(\eword,\emptyset,\eword)\notin \xi$ for all $\xi\in\ftight$. It follows that $P_{\emptyset}=0$.
	
	Now, for $A,B\in\acf$ and $\xi\in\ftight$
	\[P_AP_B(\xi)=[A\in\xi_0\wedge B\in\xi_0]\xi=[A\cap B\in\xi_0]\xi=P_{A\cap B}(\xi)\]
	where $[\ ]$ represents the boolean function that returns $1$ if the argument is true and $0$ otherwise; and the second equality follows from $\xi_0$ being a filter. Also, using that $\xi_0$ is an ultrafilter and therefore a prime filter, we have
	\dmal{P_{A\cup B}(\xi) & =[A\cup B\in\xi_0]\xi \\ &=([A\in\xi_0]+[B\in\xi_0]-[A\cap B\in\xi_0])\xi \\ &=(P_A+P_B-P_{A\cap B})(\xi).}
	
	Let $a,b\in\alf$. It is easy to see that $S_a^*S_a=P_{r(a)}$ and $S_a^*S_b=0$ if $a\neq b$. We check that $P_AS_a=S_aP_{r(A,a)}$. On one hand
	\dmal{P_AS_a(\xi^{\beta}) & =P_A([r(a)\in\xi_0]G_{(a)\beta}(\xi^{\beta})) \\ & =[A\in G_{(a)\beta}(\xi^{\beta})_0\wedge r(a)\in\xi_0]G_{(a)\beta}(\xi^{\beta}).}
	On the other hand
	\[S_aP_{r(A,a)}(\xi^{\beta})=S_a([r(A,a)\in\xi_0]\xi^\beta)=[r(A,a)\in\xi_0]G_{(a)\beta}(\xi^{\beta})\]
	where the last equality makes sense since $r(A,a)\subseteq r(a)$. If $r(a)\notin \xi_0$ then $r(A,a)\notin \xi_0$, so that both expressions above are zero. Suppose then that $r(a)\in\xi_0$. From  Remark \ref{remark:G-class.ft0.enough.if.not.empty} we have that $\left(G_{(a)\beta}(\xi^{\beta})\right)_1=g_{(a)\eword}(\xi_0)$; hence,
	\[A\in \left(G_{(a)\beta}(\xi^{\beta})\right)_0\Leftrightarrow r(A,a)\in \left(G_{(a)\beta}(\xi^{\beta})\right)_1\Leftrightarrow \exists C\in\xi_0:r(A,a)=C\cap r(a),\]
	so that
	\[[r(A,a)\in\xi_0]=[A\in G_{(a)\beta}(\xi^{\beta})_0\wedge r(a)\in\xi_0]\]
	and $P_AS_a=S_aP_{r(A,a)}$.
	
	For the last relation, let $A\in\acf$ be such that $0<\card{\lbf(A\dgraph^1)}<\infty$, and such that there is no $C\in\acf$ such that $\emptyset\neq C\subseteq A\cap \dgraph^0_{sink}$. We need to verify that
	\[P_A=\sum_{a\in\lbf(A\dgraph^1)}S_aP_{r(A,a)}S_a^*=P_A\sum_{a\in\lbf(A\dgraph^1)}S_aS_a^*.\]
	Since $\hs_a$ and $\hs_b$ are orthogonal if $a\neq b$, the rightmost sum above is a sum of orthogonal projections and therefore a projection itself. It is then sufficient to show that $H_A\subseteq\bigoplus_{a\in\lbf(A\dgraph^1)} H_a$. Suppose that $\xi^\beta\in H_A$, that is, $\xi^{\beta}\in\ftight$ is such that $(\eword,A,\eword)\in\xi$. That $\beta\neq\eword$ follows from the above condition on $A$ and Theorem \ref{thm:tight.filters.in.es}. Now, $(\eword,A,\eword)\in\xi$ is equivalent to $A\in\xi_0$, and in this case $\beta_1\in\lbf(A\dgraph^1)$ because $r(A,\beta_1)\in\xi_1$ and $\xi_1$ is a filter. Therefore, $\xi^\beta\in\bigoplus_{a\in\lbf(A\dgraph^1)} H_a$.
\end{proof}

\begin{proposition}
	If $(\alpha,A,\beta)\in S\setminus\{0\}$, then $S_{\alpha}P_AS_{\beta}^*\neq 0$.
\end{proposition}

\begin{proof}
	Observe that if $\alpha\in\awplus$ then for $\xi\in\ftight$
	\[S_{\alpha}(\xi^{\beta})=S_{\alpha_1}\cdots S_{\alpha_{|\alpha|}}(\xi^{\beta})=[r(\alpha)\in\xi_0]G_{(\alpha)\beta}(\xi)\]
	and
	\[S_{\alpha}^*=H_{[\alpha]\beta}(\xi).\]
	The above equalities are also true for $\alpha=\eword$ if we define $S_{\eword}=\mbox{Id}_{\hs}$.
	
	Let $(\alpha,A,\beta)\in S$ be given so that $\emptyset\neq A\subseteq r(\alpha)\cap r(\beta)$. Let us verify that $S_{\alpha}P_AS_{\beta}^*\neq 0$. Since $A\neq 0$, the set $\widetilde{\eta}=\{(\eword,C,\eword)\,|\,C\in\usetr{A}{\acf}\}$ is a filter in $E(S)$ and so it is contained in a ultrafilter $\eta^{\gamma}$ in $E(S)$, which is also an element of $\ftight$. From $A\subseteq r(\beta)$, it follows that $ \xi^{\beta\gamma}=G_{(\beta)\gamma}(\eta^{\gamma})$ is well defined. Then $\eta=H_{[\beta]\gamma}(\xi)$ and $A\in H_{[\beta]\gamma}(\xi)_0$. It follows that
	\dmal{S_{\alpha}P_AS_{\beta}^*(\xi) &=S_{\alpha}P_A(H_{[\beta]\gamma}(\xi))\\
		&=S_{\alpha}(\eta)\\
		&=G_{(\alpha)\gamma}(\eta)\neq 0}
	where the third equality follows from $A\subseteq r(\alpha)$ so that $G_{(\alpha)\gamma}(\eta)$ is a well defined element of $\ftight$.
\end{proof}

Joining the results from this section we obtain the following:

\begin{theorem}\label{thm:representation.general.case}
	Let $\lspace$ be a weakly left-resolving labelled space whose accommodating family $\acf$ is closed under relative complements. There exists a  representation of $C^*\lspace$ such that the image of $s_{\alpha}p_As_{\beta}^*$ is not zero for all $(\alpha,A,\beta)\in S\setminus\{0\}$.
\end{theorem}

\section{The diagonal C*-subalgebra}
\label{section:diagonal.c*-algebra}

Let $\lspace$ be a weakly left-resolving labelled space such that $\acf$ is closed under relative complements, and let $\Delta:=\Delta\lspace$ be the diagonal C*-subalgebra of $C^*\lspace$, as in definitions \ref{def:diagonal.subalgebra} and \ref{def:c*-algebra.labelled.space.v3}. The main goal of this section is to show that the spectrum of this C*-subalgebra, $\widehat{\Delta}$, is homeomorphic to $\ftight$.

\begin{proposition}\label{prop:filter.from.spectre}
	For each $\varphi\in\widehat{\Delta}$, the set \[\xi=\{(\alpha,A,\alpha)\in E(S)\,|\,\varphi(s_{\alpha}p_As_{\alpha}^*)=1\}\] is a tight filter. In particular, the map $\Phi:\widehat{\Delta}\to\ftight$ given by
	\[\Phi(\varphi)=\{(\alpha,A,\alpha)\in E(S)\,|\,\varphi(s_{\alpha}p_As_{\alpha}^*)=1\}\]
	is well defined.
\end{proposition}

\begin{proof}
	Observe that $(\alpha,A,\alpha)\leq (\beta,B,\beta)$ in $E(S)$ if and only if $s_{\alpha}p_As_{\alpha}^*\leq s_{\beta}p_Bs_{\beta}^*$ in $C^*\lspace$. Using that $\varphi$ is a *-homomorphism, it then follows that $\xi$ is a filter in $E(S)$. We have to prove that $\xi$ is tight. Let $\alpha$ be the labelled path associated with $\xi$.
	
	First, let us consider the case that $\xi$ is of infinite type. By Theorem \ref{thm:ultrafilters}, it is sufficient to show that $\xi_n$ is an ultrafilter in the Boolean algebra $\acfrg{\alpha_{1,n}}$ for each $n>0$. In order to establish this, observe that $(\alpha_{1,n},r(\alpha_{1,n}),\alpha_{1,n})\in\xi$, so that $\varphi(s_{\alpha_{1,n}}p_{r(\alpha_{1,n})}s_{\alpha_{1,n}})=1$. If $A\in\acfrg{\alpha_{1,n}}$, then $p_{r(\alpha_{1,n})}=p_A+p_{r(\alpha_{1,n})\setminus A}$. It follows that
	\[1=\varphi(s_{\alpha_{1,n}}p_{r(\alpha_{1,n})}s_{\alpha_{1,n}})=\varphi(s_{\alpha_{1,n}}p_As_{\alpha_{1,n}})+\varphi(s_{\alpha_{1,n}}p_{r(\alpha_{1,n})\setminus A}s_{\alpha_{1,n}})\]
	and hence $\varphi(s_{\alpha_{1,n}}p_{r(\alpha_{1,n})\setminus A}s_{\alpha_{1,n}})=1$ or $\varphi(s_{\alpha_{1,n}}p_{A}s_{\alpha_{1,n}})=1$. That means that $A\in\xi_n$ or $r(\alpha_{1,n})\setminus A\in\xi_n$, that is, $\xi_n$ is an ultrafilter.
	
	For the case that $\xi$ is of finite type, we use (ii) of Theorem \ref{thm:tight.filters.in.es}. If $|\alpha| > 0 $, the same argument as above shows that $\xi_{|\alpha|}$ is an ultrafilter. If $|\alpha|=0$, suppose by contradiction that $\xi_0$ is not an ultrafilter. Then there exists $C\in\acf\setminus\xi_0$ such that $C\cap A\neq\emptyset$ for all $A\in\xi_0$. For a fixed $A\in\xi_0$, since $A=(A\setminus C)\cup(A\cap C)$ and this union is disjoint, then
	\[1=\varphi(p_A)=\varphi(p_{A\setminus C}+p_{A\cap C})=\varphi(p_{A\setminus C})+\varphi(p_{A\cap C})=\varphi(p_{A\setminus C}),\]
	where the last equality holds because $A\cap C\subseteq C\notin\xi_0$. Hence $A\setminus C\in\xi_0$, but $(A\setminus C)\cap C=\emptyset$ which is a contradiction. So, in all cases, $\xi_{|\alpha|}$ is an ultrafilter. Assume now that $\xi$ is not tight; then, by (ii) of Theorem \ref{thm:tight.filters.in.es}, there exists $A\in\xi_{|\alpha|}$ such that $\lbf(A\dgraph^1)$ is finite and there is no $B\in\acfrg{\alpha}$ with $\emptyset\neq B\subseteq A\cap \dgraph^0_{sink}$. In particular $A\cap \dgraph^0_{sink}=\emptyset$, so that $\card{\lbf(A\dgraph^1)}>0$ and
	\[p_A=\sum_{b\in\lbf(A\dgraph^1)}s_bp_{r(A,b)}s_b^*\]
	holds. Since $(\alpha,A,\alpha)\in\xi$,
	\[1=\varphi(s_{\alpha}p_As_{\alpha}^*)=\sum_{b\in\lbf(A\dgraph^1)}\varphi(s_{\alpha}s_bp_{r(A,b)}s_b^*s_{\alpha}^*),\]
	which implies that $(\alpha b,r(A,b),\alpha b)\in\xi$ for some $b\in \lbf(A\dgraph^1)$; but this contradicts the fact that $\alpha$ is the word associated to $\xi$. Therefore, $\xi$ is tight.
\end{proof}

To construct the inverse of $\Phi$ from the above proposition, we have to show that if $\xi\in\ftight$ then there exists an element $\varphi\in\widehat{\Delta}$ such that
\[\varphi(s_{\alpha}p_As_{\alpha}^*)=[(\alpha,A,\alpha)\in\xi].\]
We would like to simply extend the above expression linearly to $\text{span}\{s_\alpha p_A s_\alpha^* \ | \ \alpha\in\awstar \ \mbox{and} \ A\in\acfrg{\alpha}\}$ but, in doing so, care must be taken to ensure the result is indeed a well-defined linear map. In order to show that this can be done, and that the resulting map extends to an element of $\widehat{\Delta}$, we control the norm of a finite linear combination on elements of the form $s_{\alpha}p_As_{\alpha}^*$ by rewriting the sum as a finite linear combination of orthogonal projections. We follow some of the ideas of \cite{MR3119197}.

\begin{lemma}\label{lemma:projection.q}
	Let $F\subseteq E(S)\setminus\{0\}$ be a finite set such that for all $u,v\in F$, $uv=0$, $u\leq v$ or $v\leq u$. For each $u=(\alpha,A,\alpha)\in F$, define
	\[q_u^F=q_{(\alpha,A,\alpha)}^F:=s_{\alpha}p_As_{\alpha}^*\prod_{\mathclap{\substack{(\beta,B,\beta)\in F \\ (\beta,B,\beta)< (\alpha,A,\alpha)}}}(s_{\alpha}p_As_{\alpha}^*-s_{\beta}p_Bs_{\beta}^*).\]
	Then, for all $u,v\in F$ with $u\neq v$, the projections $q^F_u$ and $q^F_v$ are mutually orthogonal projections in $\mathrm{span}\{s_{\beta}p_Bs_{\beta}^*\,|\,(\beta,B,\beta)\in F\}$. Also for $(\alpha,A,\alpha)\in F$
	\begin{equation}\label{eqn:sapasastar}
		s_{\alpha}p_As_{\alpha}^*=\sum_{\mathclap{\substack{u\in F \\ u\leq(\alpha,A,\alpha)}}}q_{u}^F.
	\end{equation}
\end{lemma}

\begin{proof}
	For all $u\in F$,  $q_u^F$ is a product of commuting projections and therefore a projection in $\text{span}\{s_{\beta}p_Bs_{\beta}^*\,|\,(\beta,B,\beta)\in F\}$. Let $u=(\alpha,A,\alpha)$ and $v=(\beta,B,\beta)$ be elements of $F$ such that $u\neq v$. If $uv=0$, then $s_{\alpha}p_As_{\alpha}^*s_{\beta}p_Bs_{\beta}^*=0$, so that $q_u^Fq_v^F=0$. If $u< v$, then $(s_{\beta}p_Bs_{\beta}^*-s_{\alpha}p_As_{\alpha}^*)$ is a factor of $q_v^F$. Since $s_{\alpha}p_As_{\alpha}^*$ is a factor of $q_u^F$ and $s_{\alpha}p_As_{\alpha}^*s_{\beta}p_Bs_{\beta}^*=s_{\alpha}p_As_{\alpha}^*$, we have that $q_u^Fq_v^F=0$. The case $v< u$ is analogous.
	
	To prove (\ref{eqn:sapasastar}), we use induction on $\card{F}$. The result is immediate if $\card{F}=1$. Let $n > 1$ and suppose that the result is true for all $F$ with $\card{F}<n$. Let $F\subseteq E(S)$ be as in the hypothesis of the lemma, and with $\card{F}=n$. Chose a minimal element $(\gamma,C,\gamma)$ in $F$ and define $G=F\setminus\{(\gamma,C,\gamma)\}$. Since $(\gamma,C,\gamma)$ is minimal in $F$ then
	\[\sum_{\mathclap{\substack{u\in F \\ u\leq(\gamma,C,\gamma)}}}q_{u}^F=q_{(\gamma,C,\gamma)}^F=s_{\gamma}p_Cs_{\gamma}^*,\] that is, (\ref{eqn:sapasastar}) holds for $(\gamma,C,\gamma)$.
	
	Observe that, for a given $(\alpha,A,\alpha)\in G$,
	\[q_{(\alpha,A,\alpha)}^F=\left\{\begin{array}{ll}
	q_{(\alpha,A,\alpha)}^G, & \text{if }(\alpha,A,\alpha)(\gamma,C,\gamma)=0; \\
	q_{(\alpha,A,\alpha)}^G-q_{(\alpha,A,\alpha)}^Gs_{\gamma}p_Cs_{\gamma}^*, & \text{if }(\gamma,C,\gamma)\leq (\alpha,A,\alpha).
	\end{array}\right.\]
	Indeed, the above equality is trivially true if $(\alpha,A,\alpha)(\gamma,C,\gamma)=0$ and, in the case $(\gamma,C,\gamma)\leq (\alpha,A,\alpha)$, it holds since
	
	\dmal{q_{(\alpha,A,\alpha)}^F &=s_{\alpha}p_As_{\alpha}^*\prod_{\mathclap{\substack{(\beta,B,\beta)\in F \\ (\beta,B,\beta)< (\alpha,A,\alpha)}}}(s_{\alpha}p_As_{\alpha}^*-s_{\beta}p_Bs_{\beta}^*) \\ &= s_{\alpha}p_As_{\alpha}^*(s_{\alpha}p_As_{\alpha}^*-s_{\gamma}p_Cs_{\gamma}^*)\prod_{\mathclap{\substack{(\beta,B,\beta)\in G \\ (\beta,B,\beta)< (\alpha,A,\alpha)}}}(s_{\alpha}p_As_{\alpha}^*-s_{\beta}p_Bs_{\beta}^*) \\
		&= q_{(\alpha,A,\alpha)}^G-q_{(\alpha,A,\alpha)}^Gs_{\gamma}p_Cs_{\gamma}^*.}
	
	Now, if $(\alpha,A,\alpha)(\gamma,C,\gamma)=0$, then
	\[\sum_{\mathclap{\substack{u\in F \\ u\leq(\alpha,A,\alpha)}}}q_{u}^F=\sum_{\mathclap{\substack{u\in G \\ u\leq(\alpha,A,\alpha)}}}q_{u}^G=s_{\alpha}p_As_{\alpha}^*,\]
	where the last equality follows from the induction hypothesis. If $(\gamma,C,\gamma)\leq(\alpha,A,\alpha)$, then
	\dmal{\sum_{\mathclap{\substack{u\in F \\ u\leq(\alpha,A,\alpha)}}}q_{u}^F &=s_{\gamma}p_Cs_{\gamma}^*+\sum_{\mathclap{\substack{u\in G \\ u\leq(\alpha,A,\alpha)}}}q_{u}^F \\
		&=s_{\gamma}p_Cs_{\gamma}^*+\sum_{\mathclap{\substack{u\in G \\ u\leq(\alpha,A,\alpha)}}}(q_{u}^G-q_{u}^Gs_{\gamma}p_Cs_{\gamma}^*) \\
		&=s_{\gamma}p_Cs_{\gamma}^*+s_{\alpha}p_As_{\alpha}^*-s_{\alpha}p_As_{\alpha}^*s_{\gamma}p_Cs_{\gamma}^*=s_{\alpha}p_As_{\alpha}^*,}
	where the second equality follows since $q_u^F=q_{u}^G-q_{u}^Gs_{\gamma}p_Cs_{\gamma}^*$ even when $u\cdot(\gamma,C,\gamma)=0$ and the third equality follows from the induction hypothesis. Thus, (\ref{eqn:sapasastar}) holds.
\end{proof}

\begin{lemma}\label{lemma:correct.subset}
	For all finite $F\subseteq E(S)\setminus\{0\}$, there exists $F'\subseteq E(S)\setminus\{0\}$ such that $F'$ satisfies the hypothesis of Lemma \ref{lemma:projection.q} and the following conditions:
	\begin{enumerate}[(i)]
		\item for all $(\alpha,A,\alpha)\in F$, there exist $(\alpha,A_1,\alpha),\ldots,(\alpha,A_n,\alpha)\in F'$ such that $A$ is the union of $A_1,\ldots,A_n$;
		\item the labelled paths that appear in elements of $F'$ are the same as those that appear in elements of $F$;
		\item\label{lemma:item.fprime.empty.intersections} if $(\alpha,A,\alpha),(\alpha,B,\alpha)\in F'$ and $A\neq B$ then $A\cap B=\emptyset$.
	\end{enumerate}
\end{lemma}

\begin{proof}
	For each finite $F\subseteq E(S)\setminus\{0\}$, define $m=\max\{|\alpha|\,|\,(\alpha,A,\alpha)\in F\}$. We prove the lemma by induction on $m$. If $m=0$, then $F=\{(\eword,B_1,\eword),\ldots,(\eword,B_l,\eword)\}$. Define
	\[\la{I}=\left\{\bigcap_{i\in I_1}B_i\setminus\bigcup_{j\in I_2}B_j\,|\,I_1\cup I_2=\{1,\ldots,l\},I_1\cap I_2=\emptyset\right\}\setminus\{\emptyset\}\]
	and $F'=\{(\eword,B,\eword)\,|\,B\in\la{I}\}$. Clearly, $F'$ satisfies the conditions in the statement.
	
	For $m>0$, suppose that the result is true for all finite $G\subseteq E(S)$ with $\max\{|\alpha|\,|\,(\alpha,A,\alpha)\in G\}<m$. Let us write $F=G_1\cup G_2$ where $G_1=\{(\alpha,A,\alpha)\in F\,|\,|\alpha|<m\}$ and $G_2=\{(\alpha,A,\alpha)\in F\,|\,|\alpha|=m\}$. By the induction hypothesis there exists $G'_1$ associated to $G_1$ as in the statement. Denote by $L_2$ the set of all $\alpha\in\awstar$ such that $(\alpha,A,\alpha)\in G_2$ for some $A\in\acf$. Fix $\alpha \in L_2$, define
	\[\la{J}_{\alpha}=\{r(A,\alpha'')\,|\,(\alpha',A,\alpha')\in G'_1\cup G_2\text{ and }\alpha'\alpha''=\alpha\}\]
	and consider $\la{I}_{\alpha}$ constructed from $\la{J}_{\alpha}$ as $\la{I}$ was constructed from $\{B_1,\ldots,B_l\}$ in the case $m=0$. Finally, define $F'_2=\cup_{\alpha\in L_2}\{(\alpha,B,\alpha)\in E(S)\,|\,B\in\la{I}_{\alpha}\}$ and $F'=G'_1\cup F'_2$, observing that $0\notin F'$.
	
	Let $u=(\alpha,A,\alpha)$ and $v=(\beta,B,\beta)$ be elements of $F'$. If $u$ and $v$ are both elements of either $G'_1$ or both elements of $F'_2$ then, by the definitions of $G'_1$ and $F'_2$, $u$ and $v$ are such that $uv=0$, $u\leq v$ or $v\leq u$. Suppose, then, that $u\in F'_2$ and $v\in G'_1$ so that $|\beta|<|\alpha|$. If $\alpha$ and $\beta$ are not comparable then $uv=0$. Otherwise, $\alpha=\beta\alpha'$ and, by the construction of $\la{I}_{\alpha}$, $A\subseteq r(B,\alpha')$ or $A\cap r(B,\alpha')=\emptyset$ (observe that $r(B,\alpha')\in\la{J}_{\alpha}$). In this case $u\leq v$ or $uv=0$.
	
	The other conditions from the statement are easily verified.
\end{proof}

\begin{remark}
	If $\xi$ is a filter in $E(S)$ and $F$ is a set satisfying the hypothesis of Lemma \ref{lemma:projection.q}, then $\xi \cap F$ is totally ordered because for all $u,v\in\xi$, $uv\neq 0$.
\end{remark}

\begin{lemma}
	Let $\xi\in\ftight$ and a finite $F\subseteq E(S)\setminus\{0\}$ be such that $\xi\cap F\neq\emptyset$ and for all $u,v\in F$, $uv=0$, $u\leq v$ or $v\leq u$. Let also $w=\min(\xi\cap F)$. Then there exists a non-zero $z\in E(S)$ such that $z\leq w$ and $zu=0$ for all $u\in F$ with $u<w$.
\end{lemma}

\begin{proof}
	The result is trivial if there is no $u\in F$ with $u<w$ (take $z=w$). Suppose, then, that there exists at least one such $u$. Let $\alpha$ be the word associated to $\xi$. Since $w\in\xi$, $w=(\alpha_{1,l},A,\alpha_{1,l})$ for some $l\geq 0$ and $A\in \acf$. We consider the cases given by Theorem \ref{thm:tight.filters.in.es}.
	
	Case (i): $\xia$ is of infinite type. Let $n$ be the greatest of all lengths of labelled paths $\beta$ that appears in an element $v=(\beta,B,\beta)\in F$ with $v<w$ and observe that $l\leq n$. For an element $u\in F$ with $u<w$ of the form $u=(\alpha_{1,m},B,\alpha_{1,m})$, we have that $u\notin \xi$ and hence $r(B,\alpha_{m+1,n})\notin \xi_n$ since $\{\xi_n\}_n$ is a complete family. Since $\xi_n$ is an ultrafilter, there exists $C_u\in\xi_n$ such that $C_u\cap r(B,\alpha_{m+1,n})=\emptyset$. Let $C=\cap_{u\in F,u<w}C_u\in\xi_n$ and define $z=(\alpha_{1,n},C\cap r(A,\alpha_{l+1,n}),\alpha_{1,n})$, which is non-zero because $C\cap r(A,\alpha_{l+1,n})\in\xi_n$. Then $z\leq w$ and it is easily verified that $zu=0$ for all $u\in F$ with $u<w$.
	
	Case (ii): $\xia$ is of finite type. Using a similar argument as the previous case we find $C\in\xi_{|\alpha|}$ such that for all $u\in F$ with $u<w$ of the form $=(\alpha_{1,m},B,\alpha_{1,m})$, we have that $C\cap r(B,\alpha_{m+1,|\alpha|})=\emptyset$. Define $D=C\cap r(A,\alpha_{l+1,|\alpha|})\in\xi_{|\alpha|}$.
	
	Case (ii)(a): $\lbf(D\dgraph^1)$ is infinite. Choose $b\in\lbf(D\dgraph^1)$ be a letter that is different from $\beta_{|\alpha|+1}$ for all labelled paths $\beta$ such that $|\beta|\geq|\alpha|+1$ and that it appears in an element $v=(\beta,B,\beta)\in F$. Define $z=(\alpha b,r(D,b),\alpha b)$. By construction this $z$ satisfies all of the conditions in the statement.
	
	Case (ii)(b): there exists $G\in\acfra$ such that $\emptyset\neq G\subseteq D\cap\dgraph^0_{sink}$. Define $z=(\alpha,G,\alpha)$ so that $z$ is non-zero and $z\leq w$. Let $u=(\beta,B,\beta)\in F$ be such that $u<w$. If $\beta$ is not comparable with $\alpha$ then it is immediate that $zu=0$. If $\beta$ is a beginning of $\alpha$ then $zu=0$ by the construction of $z$. Finally, if $\beta=\alpha\gamma$ for some $\gamma\in\awplus$, ie $\gamma\neq\eword$, then $zu=0$ because $r(G,\gamma)=\emptyset$.
\end{proof}

\begin{remark}\label{remark:q.neq.zero}
	Let $\xi\in\ftight$ and $F\subseteq E(S)\setminus\{0\}$ be a finite subset as in the previous Lemma, and let $w=\min(\xi\cap F)$. We claim that $q_w^F\neq 0$. Indeed, if $z=(\gamma,C,\gamma)$ is as in the previous Lemma, then from the definition of $q_w^F$ in Lemma \ref{lemma:projection.q} and from $z\leq w$, it follows that $s_{\gamma}p_Cs^*_{\gamma}q_w^F=s_{\gamma}p_Cs^*_{\gamma}$. From Theorem \ref{thm:representation.general.case}, $s_{\gamma}p_Cs^*_{\gamma}\neq 0$ so that $q_w^F \neq 0$. 
\end{remark}

\begin{lemma}\label{lemma:norm.element.diagonal}
	Let $\xi\in\ftight$ be given and let $F\subseteq E(S)\setminus\{0\}$ be a finite set. For each $u\in F$, let $\lambda_u\in\mathbb{C}\setminus\{0\}$. Then
	\[\left|\sum_{u\in F\cap\xi}\lambda_u\right|\leq\left\|\sum_{(\alpha,A,\alpha)\in F}\lambda_{(\alpha,A,\alpha)}s_{\alpha}p_As^*_{\alpha}\right\|.\]
\end{lemma}

\begin{proof}
	The result is trivial if $F\cap\xi=\emptyset$, so suppose $F\cap\xi\neq\emptyset$. Let $F'$ be the set constructed from $F$ as in Lemma \ref{lemma:correct.subset}. 	Observe that $F'\cap\xi\neq\emptyset$; indeed, for a given $(\alpha,A,\alpha)\in F\cap\xi$ there exist $(\alpha,B_1,\alpha),\ldots,(\alpha,B_m,\alpha)\in F'$ such that $A$ is the disjoint union of $B_1,\ldots,B_m$. Since $A\in\xi_{|\alpha|}$ and $\xi_{|\alpha|}$ is a prime filter, there exists $i_0\in\{1,\ldots,m\}$ such that $B_{i_0}\in\xi_{|\alpha|}$ and therefore $(\alpha,B_{i_0},\alpha)\in F'\cap\xi$. If follows that $F'\cap\xi\neq\emptyset$, as claimed.
	
	For each $(\gamma,C,\gamma)\in F'$ and $(\alpha,A,\alpha)\in F$ with $(\alpha,A,\alpha)\geq (\gamma,C,\gamma)$, define
	\[\eta_{(\gamma,C,\gamma),(\alpha,A,\alpha)}=\card{\{(\alpha,B,\alpha)\in F'\,|\,(\alpha,A,\alpha)\geq(\alpha,B,\alpha)\geq(\gamma,C,\gamma)\}}.\]
	
	Let us prove that $\eta_{(\gamma,C,\gamma),(\alpha,A,\alpha)}\leq 1$: let $(\alpha,B,\alpha),(\alpha,B',\alpha)\in F'$ be such that $(\alpha,A,\alpha)\geq(\alpha,B,\alpha)\geq (\gamma,C,\gamma)$ and $(\alpha,A,\alpha)\geq(\alpha,B',\alpha)\geq (\gamma,C,\gamma)$. Then, in the semilattice $E(S)$, $(\alpha,B,\alpha)(\alpha,B',\alpha)=(\alpha,B\cap B',\alpha)\geq (\gamma,C,\gamma)\neq 0$ and hence $B\cap B'\neq\emptyset$. Therefore, by property  (\ref{lemma:item.fprime.empty.intersections}) of $F'$ in Lemma \ref{lemma:correct.subset}, we must have $B=B'$. It follows that $\eta_{(\gamma,C,\gamma),(\alpha,A,\alpha)}\leq 1$.

	Let $w=\min (F'\cap\xi)$. If $(\alpha,B_{i_0},\alpha)$ is as in the first paragraph of the proof, then $w\leq (\alpha,B_{i_0},\alpha)\leq (\alpha,A,\alpha)$ and so $\eta_{w,(\alpha,A,\alpha)}\geq 1$, hence $\eta_{w,(\alpha,A,\alpha)}=1$.
	
	Therefore, 
	\dmal{\sum_{(\alpha,A,\alpha)\in F} \lambda_{(\alpha,A,\alpha)}s_{\alpha}p_As_{\alpha}^* 
		&\stackrel{(1)}{=}\sum_{(\alpha,A,\alpha)\in F}\lambda_{(\alpha,A,\alpha)}s_{\alpha}\left(\sumtl{(\alpha,B,\alpha)\in F'}{B\subseteq A}{p_B}\right) s_{\alpha}^* \\
		&=\sum_{(\alpha,A,\alpha)\in F}\lambda_{(\alpha,A,\alpha)}\sumtl{(\alpha,B,\alpha)\in F'}{B\subseteq A}{s_{\alpha}p_Bs^*_{\alpha}} \\
		&\stackrel{(2)}{=}\sum_{(\alpha,A,\alpha)\in F}\lambda_{(\alpha,A,\alpha)}\sumtl{(\alpha,B,\alpha)\in F'}{B\subseteq A}{}\left(\sumtl{(\gamma,C,\gamma)\in F'}{(\gamma,C,\gamma)\leq(\alpha,B,\alpha)}{q^{F'}_{(\gamma,C,\gamma)}}\right) \\
		&\stackrel{(3)}{=}\sum_{(\alpha,A,\alpha)\in F}\lambda_{(\alpha,A,\alpha)}\sumtl{(\gamma,C,\gamma)\in F'}{(\gamma,C,\gamma)\leq(\alpha,A,\alpha)}{\eta_{(\gamma,C,\gamma),(\alpha,A,\alpha)}q^{F'}_{(\gamma,C,\gamma)}} \\
		&\stackrel{(4)}{=}\sum_{(\gamma,C,\gamma)\in F'}\left(\sumtl{(\alpha,A,\alpha)\in F}{(\alpha,A,\alpha)\geq(\gamma,C,\gamma)}{\eta_{(\gamma,C,\gamma),(\alpha,A,\alpha)}\lambda_{(\alpha,A,\alpha)}}\right)q^{F'}_{(\gamma,C,\gamma)}.}
	The equalities above are justified as follows: (1) is a consequence of Lemma \ref{lemma:correct.subset}, (2) follows from Lemma \ref{lemma:projection.q}, (3) is due to the definition of $\eta_{(\gamma,C,\gamma),(\alpha,A,\alpha)}$, and (4) is true since the sums on both sides are over all pairs $(\alpha,A,\alpha)\in F$, $(\gamma,C,\gamma)\in F'$ such that $(\gamma,C,\gamma)\leq(\alpha,A,\alpha)$.

	Hence
	\dmal{\left\|\sum_{(\alpha,A,\alpha)\in F} \lambda_{(\alpha,A,\alpha)}s_{\alpha}p_As_{\alpha}^*\right\| &=\left\|\sum_{(\gamma,C,\gamma)\in F'}\left(\sumtl{(\alpha,A,\alpha)\in F}{(\alpha,A,\alpha)\geq(\gamma,C,\gamma)}{\eta_{(\gamma,C,\gamma),(\alpha,A,\alpha)}\lambda_{(\alpha,A,\alpha)}}\right)q^{F'}_{(\gamma,C,\gamma)}\right\| \\
		&\stackrel{(5)}{=}\max_{\substack{(\gamma,C,\gamma)\in F' \\ q^{F'}_{(\gamma,C,\gamma)}\neq 0}}\left|\sumtl{(\alpha,A,\alpha)\in F}{(\alpha,A,\alpha)\geq(\gamma,C,\gamma)}{{\eta_{(\gamma,C,\gamma),(\alpha,A,\alpha)}\lambda_{(\alpha,A,\alpha)}}}\right| \\
		&\stackrel{(6)}{\geq}\left|\sumtl{(\alpha,A,\alpha)\in F}{(\alpha,A,\alpha)\geq w}{\eta_{w,(\alpha,A,\alpha)}\lambda_{(\alpha,A,\alpha)}}\right| \\
		&\stackrel{(7)}{=}\left|\sum_{(\alpha,A\alpha)\in F\cap\xi}\eta_{w,(\alpha,A,\alpha)}\lambda_{(\alpha,A,\alpha)}\right|\\
		&\stackrel{(8)}{=}\left|\sum_{(\alpha,A,\alpha)\in F\cap\xi}\lambda_{(\alpha,A,\alpha)}\right|,}
	where (5) is due to the fact that the projections $q^{F'}_{(\gamma,C,\gamma)}$ are pairwise orthogonal (Lemma \ref{lemma:projection.q}), (6) follows from Remark \ref{remark:q.neq.zero}, (7) is true since $\xi$ is a filter and (8) is a consequence of $\eta_{w,(\alpha,A,\alpha)}=1$.
\end{proof}

\begin{proposition}\label{prop:spectre.from.filter}
	For each $\xi\in\ftight$ there is a unique $\varphi\in\widehat{\Delta}$ such that $\varphi(s_{\alpha}p_As^*_{\alpha})=[(\alpha,A,\alpha)\in\xi]$ for all $(\alpha,A,\alpha)\in E(S)$.
\end{proposition}

\begin{proof}
	An element $x\in\text{span}\{s_{\alpha}p_As^*_{\alpha}\,|\,(\alpha,A,\alpha)\in E(S)\}$ can be written as
	\[x=\sum_{(\alpha,A,\alpha)\in F}\lambda_{(\alpha,A,\alpha)}s_{\alpha}p_As^*_{\alpha}\]
	for some finite $F\subseteq E(S)$. By Lemma \ref{lemma:norm.element.diagonal},
	\[\varphi(x)=\sum_{(\alpha,A,\alpha)\in F\cap\xi}\lambda_{(\alpha,A,\alpha)}\]
	gives a well defined continuous linear map from $\text{span}\{s_{\alpha}p_As^*_{\alpha}\,|\,(\alpha,A,\alpha)\in E(S)\}$ into $\mathbb{C}$. It is easily verified that $\varphi$ preservers products so that it extends to an element $\varphi\in\widehat{\Delta}$ that satisfies the equality from the statement. The uniqueness is immediate.
\end{proof}

\begin{theorem}
	Let $\lspace$ be a weakly left-resolving labelled space such that $\acf$ is closed under relative complements. Then, there exists a homeomorphism between $\ftight$ and $\widehat{\Delta}$.
\end{theorem}

\begin{proof}
	Putting together Propositions \ref{prop:filter.from.spectre} and \ref{prop:spectre.from.filter}, we have a bijection $\Phi:\widehat{\Delta}\to\ftight$ given by
	\[\Phi(\varphi)=\{(\alpha,A,\alpha)\in E(S)\,|\,\varphi(s_{\alpha}p_As_{\alpha}^*)=1\}.\]
	
	Since the topologies on $\widehat{\Delta}$ and $\ftight$ are both given by pointwise convergence, it follows that $\Phi$ is continuous. A standard $\varepsilon/3$ argument shows that $\Phi^{-1}$ is continuous.		
\end{proof}

\section{An Example}
\label{section:example.definitions.different}

In this section, we present an example of a labelled space whose C*-algebra given by Definition \ref{def:c*-algebra.labelled.space.v3} is a non-trivial quotient of the C*-algebra considered in a preprint of \cite{ETS:9992065} (see Remark \ref{remark.condition.iv}). To see this, we construct a representation of the C*-algebra based on this alternative definition such that the images of the generating projections and partial isometries do not satisfy the additional relations given in Definition \ref{def:c*-algebra.labelled.space.v3}.

From now on, fix a labelled space $\lspace$ whose labelled graph is left-resolving. Let $Y$ be an infinite set such that $\card{Y}\geq\max\{\card{\dgraph^0},\card{\dgraph^1}\}$. For each $e\in\dgraph^1$, let $E_e$ be a copy of $Y$ and, for each $v\in\dgraph^0_{sink}$, let $D_v$ be also a copy of $Y$, so that all copies of $Y$ are pairwise disjoint. For each $v\in\dgraph^0\setminus\dgraph^0_{sink}$, define\footnote{Here, we use the square union symbol to stress that the union is disjoint.} $D_v = \bigsqcup_{e\in s^{-1}(v)}E_e$. It is easy to see that the $D_v$ are pairwise disjoint. Now, for each $e\in\dgraph^1$, choose a bijection $\newf{h_e}{D_{r(e)}}{E_e}$.\footnote{In the case $D_{r(e)}=E_e$ (that is, $s(e)=r(e)$ and $e$ is the unique arrow with source $r(e)$), $h_e$ cannot be the identity function.} For each letter $a\in\alf$, the union $\cup_{e\in \lbf^{-1}(a)}D_{r(e)}$ is disjoint, since the labelled graph is left-resolving. Thus, we can define a bijection $\newf{h_a}{\bigsqcup_{e\in \lbf^{-1}(a)}D_{r(e)}}{\bigsqcup_{e\in \lbf^{-1}(a)}E_{e}}$ by gluing the functions $h_e$, where $\lbf(e)=a$. Finally, set $X=\bigsqcup_{v\in\dgraph^0}D_v$.

Consider the Hilbert space $\ell^2(X)$ e let $\{\delta_x\}_{x\in X}$ be its canonical basis.

For each letter $a\in\alf$, define
\[\newfd{S_a}{\ell^2(X)}{\ell^2(X)}{\delta_x}{[x\in\bigsqcup_{e\in \lbf^{-1}(a)}D_{r(e)}]\delta_{h_a(x)},}\]
recalling that $[\ ]$ represents the boolean function that returns $1$ if the argument is true and $0$ otherwise. In this way, $S_a$ is a partial isometry with $\ell^2(\bigsqcup_{e\in \lbf^{-1}(a)}D_{r(e)})$ as its initial space and $\ell^2(\bigsqcup_{e\in \lbf^{-1}(a)}E_{e})$ as its final space.

For each $A\subseteq\dgraph^0$, define
\[\newfd{P_A}{\ell^2(X)}{\ell^2(X)}{\delta_x}{[x\in\bigsqcup_{v\in A}D_{v}]\delta_{x}.}\]
Thus, $P_A$ is the projection onto $\ell^2(\bigsqcup_{v\in A}D_{v})$.

\begin{proposition}\label{prop:representation.of.def.v2} The operators $\{S_a\}_{a\in\alf}$ and $\{P_A\}_{A\in\acf}$ satisfy the conditions of Definition \ref{def:c*-algebra.labelled.space.v3}, replacing item (iv) with
\begin{enumerate}[(iv)]
	\item For every $A\in\acf$ such that $0<\card{\lbf(A\dgraph^1)}<\infty$ and $A\cap\dgraph^0_{sink}=\emptyset$,
	\[p_A=\sum_{a\in\lbf(A\dgraph^1)}s_ap_{r(A,a)}s_a^*\]
\end{enumerate}
(see Remark \ref{remark.condition.iv}).
\end{proposition}

\begin{proof}
	We already know that $S_a$ is a partial isometry and $P_A$ is a projection. We only show items (ii) and (iv), since (i) and (iii) are trivial.
	
	Observe that
	\[P_AS_a(\delta_x) = [x\in\bigsqcup_{e\in \lbf^{-1}(a)}D_{r(e)}][h_a(x)\in\bigsqcup_{v\in A}D_v]\delta_{h_a(x)}\]
	and
	\[S_aP_{r(A,a)}(\delta_x) = [x\in\bigsqcup_{v\in r(A,a)}D_{v}][x\in\bigsqcup_{e\in \lbf^{-1}(a)}D_{r(e)}]\delta_{h_a(x)}.\]
	Thus, to see (ii), it suffices to show that $[h_a(x)\in\bigsqcup_{v\in A}D_v]=[x\in\bigsqcup_{v\in r(A,a)}D_{v}]$ whenever $[x\in\bigsqcup_{e\in \lbf^{-1}(a)}D_{r(e)}]=1$. Indeed, if $x\in\bigsqcup_{v\in r(A,a)}D_{v}$, then there exists $e\in\lbf^{-1}(a)$ with $s(e)\in A$ such that $x\in D_{r(e)}$. Therefore, $h_a(x)\in E_e\subseteq D_{s(e)} \subseteq \bigsqcup_{v\in A}D_v$. On the other hand, since $x\in\bigsqcup_{e\in \lbf^{-1}(a)}D_{r(e)}$, then there exists $e\in\lbf^{-1}(a)$ such that $x\in D_{r(e)}$ and, hence, $h_a(x)\in E_e$. Thus, if $h_a(x)\in\bigsqcup_{v\in A}D_v$, then $E_e\subseteq \bigsqcup_{v\in A}D_v$, which says that $s(e)\in A$. In other words, $r(e)\in r(A,a)$, showing that $x\in D_{r(e)}\subseteq \bigsqcup_{v\in r(A,a)}D_{v}$.
	
	Let $A\in\acf$ be such that $\lbf(A\dgraph^1)$ is finite and $A\cap\dgraph^0_{sink}=\emptyset$. Now that (ii) has been established, proving (iv) is equivalent to showing that $S_aS_a^*$ and $S_bS_b^*$ are orthogonal projections for $a,b\in\lbf(A\dgraph^1)$ with $a\neq b$, and \[P_A\leq\sum_{a\in\lbf(A\dgraph^1)}S_aS_a^*.\]
	By item (iii), it is clear that $S_aS_a^*$ and $S_bS_b^*$ are orthogonal if $a\neq b$. The operator $P_A$ is the projection onto $\ell^2(\bigsqcup_{v\in A}D_{v})$ and the operator $\sum_{a\in\lbf(A\dgraph^1)}S_aS_a^*$ is the projection onto $\ell^2(\bigsqcup_{a\in\lbf(A\dgraph^1)}\bigsqcup_{e\in\lbf^{-1}(a)}E_{e})$. Since $A\cap\dgraph^0_{sink}=\emptyset$, then $\bigsqcup_{v\in A}D_{v} = \bigsqcup_{v\in A}\bigsqcup_{e\in s^{-1}(v)}E_e$ and it is clearly contained in $\bigsqcup_{a\in\lbf(A\dgraph^1)}\bigsqcup_{e\in\lbf^{-1}(a)}E_e$. Therefore,
	$$P_A\leq\sum_{a\in\lbf(A\dgraph^1)}S_aS_a^*.$$
\end{proof}

This Proposition ensures that there exists a homomorphism from $C^*\lspace$ (with the alternative item (iv)) to $\mathcal{B}(\ell^2(X))$ which sends $s_a$ to $S_a$ and $p_A$ to $P_A$. To simplify the writing, we refer to Definition  \ref{def:c*-algebra.labelled.space.v3} with the alternative item (iv) as the \emph{alternative} definition.

\begin{remark}
	We observe that there are more relations in Definition \ref{def:c*-algebra.labelled.space.v3} than in the alternative. Thus, every representation of the C*-algebra given by Definition \ref{def:c*-algebra.labelled.space.v3} is also a representation of the one given by the alternative (such as the representation in Section \ref{section:representations}, for instance). The representation stated above is not, in general, a representation of the C*-algebra given by Definition \ref{def:c*-algebra.labelled.space.v3}.
\end{remark}

Now, we are ready to see the promised example.

\begin{example} \label{example:different.definitions}
	Consider the graph below.
	
	\begin{center}
		\resizebox{6cm}{!}
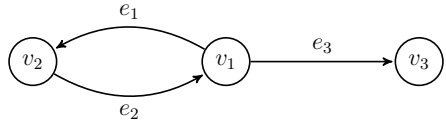
{
			\begin{tikzpicture}[->,>=stealth',shorten >=1pt,auto,thick,main node/.style={circle,draw,thick,font=\bfseries}]
			\node[main node] (1) {$v_1$};
			\node[main node] (2) [left=2.5cm of 1] {$v_2$};
			\node[main node] (3) [right=2.5cm of 1] {$v_3$};
			
			\draw[->] (2) to [bend right] node [below] {$e_2$} (1);
			\draw[->] (1) to [bend right] node [above] {$e_1$} (2);
			\draw[->] (1) to node [above] {$e_3$} (3);
			\end{tikzpicture}
		}
	\end{center}
	
	For each arrow, assign the label $a$ and consider the family \[\acf=\{\emptyset, \{v_1\}, \{v_2,v_3\}, \{v_1,v_2,v_3\}\},\] which is easily seen to be accommodating and closed under relative complements.
	
	The C*-algebra of this labelled space (under any of the two given definitions) is generated by four elements: $s_a$, $p_{\{v_1\}}$, $p_{\{v_2, v_3\}}$ and $p_{\{v_1, v_2, v_3\}}$. Clearly, using (i), (ii) and (iii), we see that $p_{\{v_1\}}+p_{\{v_2,v_3\}}=p_{\{v_1, v_2, v_3\}} = s_a^*s_a = 1$.
	
	Observe that $A=\{v_1,v_2,v_3\}$ does not satisfy $A\cap\dgraph^0_{sink}=\emptyset$, and there is no $B\in\acf$ such that $\emptyset\neq B\subseteq A\cap\dgraph^0_{sink}$. This means there is a relation in Definition \ref{def:c*-algebra.labelled.space.v3} that does not appear in the alternative definition.
	
	Now, suppose we are under Definition \ref{def:c*-algebra.labelled.space.v3}. Since $\{v_1,v_2,v_3\}$ satisfies the conditions of item (iv), we conclude that $s_as_a^*=1$. Furthermore, $\{v_1\}=r(\{v_1,v_2\},a)$, $\{v_1,v_2\} = r(\{v_1\},a)$ and $\{v_1,v_2,v_3\}=r(\{v_1,v_2,v_3\},a)$, which says that every set in $\acf$ is the relative range (by $a$) of another set in $\acf$. Thus, for every $A\in\acf$, $s_ap_A=p_Bs_a$ for some $B\in\acf$. Therefore, every element of the form $s_{\alpha}p_A s_{\alpha}^*$ is equal to $p_B$ for some $B\in\acf$, from where we conclude that $p_{\{v_1\}}$ and $p_{\{v_2,v_3\}}$ generate $\Delta\lspace$. Hence, the spectrum of $\Delta\lspace$ is a set with two points.
	
	On the other hand, under the alternative definition, we can consider the representation developed in this section applied to this labelled space to obtain $S_aS_a^*=P_{\{v_1,v_2\}}\neq P_{\{v_1,v_2,v_3\}} = 1$. In this way, we see that $p_{\{v_1\}}$, $p_{\{v_2,v_3\}}$, $p_{\{v_1,v_2,v_3\}}$ and $s_as_a^*$ are all different and non-zero. Therefore, in this case, the spectrum of $\Delta\lspace$ possesses more than two points.
	
	This shows that $C^*\lspace$ given by Definition \ref{def:c*-algebra.labelled.space.v3} is a non-trivial quotient of the C*-algebra given by the alternative definition. Furthermore, we conclude that the diagonal C*-algebras are different and, hence, the spectrum of the diagonal C*-algebra of the alternative definition is not homeomorphic to the tight spectrum of the inverse semigroup associated with the labelled space.
\end{example}

\bibliographystyle{abbrv}
\bibliography{labelledgraphs_ref}

\begin{thebibliography}{10}

\bibitem{ETS:9992065}
T.~Bates, T.~M. Carlsen, and D.~Pask.
\newblock {$C\sp *$}-algebras of labelled graphs {III}-{K}-theory computations.
\newblock {\em Ergodic Theory and Dynamical Systems}, FirstView:1--32, 4 2016.

\bibitem{MR2304922}
T.~Bates and D.~Pask.
\newblock {$C\sp *$}-algebras of labelled graphs.
\newblock {\em J. Operator Theory}, 57(1):207--226, 2007.

\bibitem{MR2542653}
T.~Bates and D.~Pask.
\newblock {$C\sp *$}-algebras of labelled graphs. {II}. {S}implicity results.
\newblock {\em Math. Scand.}, 104(2):249--274, 2009.

\bibitem{MR3231477}
T.~Bates, D.~Pask, and P.~Willis.
\newblock Group actions on labeled graphs and their {$C^*$}-algebras.
\newblock {\em Illinois J. Math.}, 56(4):1149--1168, 2012.

\bibitem{Boava2016}
G.~Boava, G.~G. de~Castro, and F.~L. Mortari.
\newblock Inverse semigroups associated with labelled spaces and their tight
  spectra.
\newblock {\em Semigroup Forum}, pages 1--28, 2016.

\bibitem{MR561974}
J.~Cuntz and W.~Krieger.
\newblock A class of {$C^{\ast} $}-algebras and topological {M}arkov chains.
\newblock {\em Invent. Math.}, 56(3):251--268, 1980.

\bibitem{MR2419901}
R.~Exel.
\newblock Inverse semigroups and combinatorial {$C\sp \ast$}-algebras.
\newblock {\em Bull. Braz. Math. Soc. (N.S.)}, 39(2):191--313, 2008.

\bibitem{MR2834773}
J.~A. Jeong and S.~H. Kim.
\newblock On simple labelled graph {$C\sp \ast$}-algebras.
\newblock {\em J. Math. Anal. Appl.}, 386(2):631--640, 2012.

\bibitem{MR2873859}
J.~A. Jeong, S.~H. Kim, and G.~H. Park.
\newblock The structure of gauge-invariant ideals of labelled graph {$C\sp
  \ast$}-algebras.
\newblock {\em J. Funct. Anal.}, 262(4):1759--1780, 2012.

\bibitem{MR1432596}
A.~Kumjian, D.~Pask, I.~Raeburn, and J.~Renault.
\newblock Graphs, groupoids, and {C}untz-{K}rieger algebras.
\newblock {\em J. Funct. Anal.}, 144(2):505--541, 1997.

\bibitem{MR1507106}
M.~H. Stone.
\newblock Postulates for {B}oolean {A}lgebras and {G}eneralized {B}oolean
  {A}lgebras.
\newblock {\em Amer. J. Math.}, 57(4):703--732, 1935.

\bibitem{MR3119197}
S.~B.~G. Webster.
\newblock The path space of a directed graph.
\newblock {\em Proc. Amer. Math. Soc.}, 142(1):213--225, 2014.

\end{thebibliography}

\end{document}